\documentclass[11pt,a4paper]{amsart}
\usepackage{
amssymb,
amsmath}

\usepackage[bookmarks,colorlinks,pagebackref]{hyperref} 

\usepackage{mathptmx}
\usepackage{helvet}
\usepackage{courier}

\usepackage[all,cmtip,2cell]{xy}

\numberwithin{equation}{section}
\theoremstyle{plain} % style de mise en pages plain: normal (théorème)
\newtheorem{proposition}{Proposition}[section]  % Proposition numérotée par section
\newtheorem{lemma}[proposition]{Lemma}
\newtheorem{corollary}[proposition]{Corollary} % idem
\newtheorem{theorem}[proposition]{Theorem} % idem

\theoremstyle{definition} % style de mise en page définition

\newtheorem{remark}[proposition]{Remark} % Remarque 

\newtheorem{example}[proposition]{Example} % Exemple numérotés par section

\newtheorem{z}[proposition]{}

\newcommand\Tor{\operatorname{Tor}}
\newcommand\Hom{\operatorname{Hom}}

\newcommand\Ext{\operatorname{Ext}}

\newcommand\Ker{\operatorname{Ker}}
\newcommand\Coker{\operatorname{Coker}}
\newcommand\im{\operatorname{Im}}
\newcommand{\id}{\operatorname{id}}
\newcommand\Mic{\operatorname{Mic}}
\newcommand\Tel{\operatorname{Tel}}
\newcommand{\xx}{\underline x}

\newcommand{\ts}{\underline t}
\newcommand{\UU}{\underline U}

\newcommand{\Cech}{{\Check {C}}_{\xx}}

\newcommand{\quism}{\stackrel{\sim}{\longrightarrow}}
\newcommand{\qism}{\stackrel{\sim}{\longrightarrow}}
\newcommand{\lqism}{\stackrel{\sim}{\longleftarrow}}

\author[P.~Schenzel]{Peter Schenzel}
\title[\v{C}ech (co-) complexes]
{\v{C}ech (co-) complexes as Koszul complexes and applications}
\address{Martin-Luther-Universit\"at Halle-Wittenberg,
Institut f\"ur Informatik, D --- 06 099 Halle (Saale), Germany}
\email{peter.schenzel@informatik.uni-halle.de}

\begin{document}

\begin{abstract} 
	In the main results of the paper it is shown that the \v{C}ech (co-) homology 
	might be considered as an appropriate Koszul (co-) homology. 
	%In particular, it 
	%follows that local cohomology is isomorphic to an appropriate Koszul cohomology.
	Let $\check{C}_{\xx}$ denote the \v{C}ech complex with respect to a system of 
	elements $\xx = x_1,\ldots,x_r$ of a commutative ring $R$. We construct a 
	bounded complex $\mathcal{L}_{\xx}$ of free $R$-modules and a 
	quasi-isomorphism $\mathcal{L}_{\xx} \qism \check{C}_{\xx}$ and isomorphisms  
	$\mathcal{L}_{\xx} \otimes_R X \cong K^{\bullet}(\xx-\UU; X[\UU^{-1}])$ 
	and $\Hom_R(\mathcal{L}_{\xx},X) \cong K_{\bullet}(\xx-\UU;X[[\UU]])$ for an $R$-complex $X$. Here 
	$\xx - \UU$ denotes the sequence of elements $x_1-U_1,\ldots,x_r-U_r$ 
	in the polynomial ring $R[\UU] = R[U_1,\ldots,U_r]$ in the variables $\UU= U_1,\ldots,U_r$ over $R$. Moreover $X[[\UU]]$ denotes the formal power series 
	complex of $X$ in $\UU$ and $X[\UU^{-1}]$ denotes the complex of inverse polynomials of $X$ in $\UU$. Furthermore $K_{\bullet}(\xx-\UU;X[[\UU]])$ resp.
	$K^{\bullet}(\xx-\UU; X[\UU^{-1}])$ denotes the corresponding Koszul complex 
	resp. the corresponding Koszul co-complex. In particular, there is a bounded $R$-free resolution of $\check{C}_{\xx}$ by a certain Koszul complex. 
	This has various 
	consequences e.g. in the case when $\xx$ is a weakly pro-regular sequence. Under this additional assumption it follows that the local cohomology 
	$H^i_{\xx R}(X)$ and the left derived functors of the completion $\Lambda_i^{\xx R}(X), i \in \mathbb{Z},$ is a certain Koszul cohomology and Koszul homology resp. This provides new approaches to the right derived functor 
	of torsion and the left derived functor of completion with various applications.	
\end{abstract}

\subjclass%[2000]
{Primary: 13D25, 13D45 ; Secondary: 13B35}

\maketitle
\begin{center}
	\textsl{Dedicated to my daughters Sara and Judith.}
\end{center}

\tableofcontents

\section{Introduction} 
Let $\xx = x_1,\ldots,x_r$ denote a system of elements of a commutative 
ring $R$. The \v{C}ech complex $\check{C}_{\xx}(R)$ (see \ref{comp-1} for its 
construction) plays an essential r\^ole in homological algebra. In the 
case $R$ is Noetherian it was shown by Grothendieck (see \cite{Ga3} and \cite{Ga2}) that there are natural isomorphisms $H^i(\check{C}_{\xx} \otimes_R M) 
\cong H^i_{\mathfrak{a}}(M)$ for all $i \geq 0$ and an $R$-module $M$, where 
$\mathfrak{a} = \xx R$ and $H^i_{\mathfrak{a}}(M)$ denotes the $i$-th 
local cohomology module of $M$ with respect to $\mathfrak{a}$. Note that 
$H^i_{\mathfrak{a}}(M)$ is the $i$-th right derived functor of the section functor 
$\Gamma_{\mathfrak{a}}(M) = \{m \in M |  \mathfrak{a}^r \cdot m = 0 \mbox{ for some } r \in \mathbb{N}\}$. 

Let $X$ denote an $R$-complex for $R$ a Noetherian ring. In the derived category  $\check{C}_{\xx}(X) = \check{C}_{\xx} \otimes_R X$ is a 
representative of ${\rm R} \Gamma_{\mathfrak{a}}(X)$ the right derived 
functor of the section functor (see e.g. \cite{ALL1} resp.  \cite{SS} for a generalization to non-Noetherian rings). Let ${\rm L} \Lambda^{\mathfrak{a}}(X)$ 
denote the left derived functor of the $\mathfrak{a}$-adic completion functor 
of $X$ in the derived category. Then there is an isomorphism 
${\rm L} \Lambda^{\mathfrak{a}}(X) \cong {\rm R} \Hom_R(\check{C}_{\xx},X)$ 
(see e.g. \cite{ALL1} resp. \cite{SS} for a generalization to the non-Noetherian case and unbounded complexes). Here we denote by $\Lambda_i^{\mathfrak{a}}(\cdot)$ the left derived functor of the $\mathfrak{a}$-adic completion $\varprojlim (R/\mathfrak{a}^n \otimes_R\cdot)$. In the case of an $R$-module these functors have been studied at first by Greenlees and May (see \cite{GM}), Lipman et al. (see \cite{ALL1}) and Simon (see \cite{Sam1}). See also the monograph \cite{SS} for further details. 

In order to have a concrete representative of 
${\rm R} \Hom_R(\check{C}_{\xx},X)$ one needs either an injective resolution 
of $X$ or a projective resolution of $\check{C}_{\xx}$. It is a bit surprisingly 
that $\check{C}_{\xx}$ is of finite projective dimension. A first resolution was 
constructed by Greenlees and May (see \cite{GM} and \ref{comp-2}). A more short resolution 
is given in \cite{Sp2} and \cite{SS} (see also \ref{comp-3}). 
Here we present a third resolution $\mathcal{L}_{\xx}$ - in some sense - minimal (see Section 3 and \ref{comp-7}). It turns out (see \ref{coh-9} (B)) that this 
resolution is given by a certain Koszul complex.

Let $R[\UU]$ denote the polynomial ring over 
$R$ in the variables $\UU = U_1,\ldots,U_r$ and let $\xx - \UU$ be the sequence 
of elements $x_1-U_1,\ldots,x_r-U_r$ in $R[\UU]$. For an $R$-complex $X$ let 
$X[[\UU]]$ denote the formal power series complex in the variables $\UU$ (see 
\ref{prel-6} for the precise definition). 

\begin{theorem} \label{int-1}
	Let $X$ denote an arbitrary $R$-complex. 
	With the prevoius notation there is an isomorphism of $R$-complexes 
	\[
	\Hom_R(\mathcal{L}_{\xx},X) \cong K_{\bullet}(\xx-\UU;X[[\UU]]),
	\]
	where $K_{\bullet}(\xx-\UU;X[[\UU]])$ denotes the Koszul complex of 
	the system of elements $\xx - \UU$ in $R[\UU]$ with respect to $X[[\UU]]$.
\end{theorem}

This is proved in \ref{hoc-1}. The Theorem \ref{int-1} is an improvement of  
\cite[8.1.6]{SS} where a quasi-isomorphism to the Koszul complex 
$K_{\bullet}(\xx-\UU;X[[\UU]])$ is shown. This extension is essential in order to prove the following results, in particular those about local cohomology.
As mentioned above $\Hom_R(\mathcal{L}_{\xx},X)$ 
is a representative of ${\rm R}\Hom_R(\check{C}_{\xx},X)$. So there 
is the following application.

\begin{corollary} \label{int-2} 
	Assume in addition that $\xx$ forms a weakly pro-regular sequence. 
	Let $\mathfrak{a} = \xx R$. Then there are isomorphisms 
	\[
	\Lambda_i^{\mathfrak{a}}(X) \cong H_i(\xx-\UU;X[[\UU]]),
	\]
	for all $i\in \mathbb{Z}$ where $H_i(\xx-\UU;X[[\UU]])$ denotes the $i$-th Koszul homology. 
\end{corollary}

This follows also by \cite[8.1.5]{SS} with a different argument. The surprising fact is that it is 
possible to compute the left derived functors of the completion in terms of a 
certain Koszul complex. For the proof and a few other isomorphisms we refer to 
\ref{weak-9}. 

In the following we shall present a dual statement for the 
\v{C}ech complex, i.e. an isomorphism to a certain Koszul co-complex. To this 
end let $R[\UU^{-1}] = R[U_1^{-1},\ldots,U_r^{-1}]$ denote the module of inverse polynomials 
over $R[\UU]$ in the sense of Macaulay (see \cite{Mfs}) and \ref{coh-1}. 
For an $R$-complex $X$ we define the complex of inverse polynomials $X[\UU^{-1}] = X[\UU] \otimes_{R[\UU]} R[\UU^{-1}]$. 

\begin{theorem} \label{int-3}
	With the previous notation there is an isomorphism of $R$-complexes 
	\[
	\mathcal{L}_{\xx}(X) \cong K^{\bullet}(\xx-\UU;X[\UU^{-1}])
	\]
	and a quasi-isomorphism 
	\[
	\check{C}_{\xx}(X) \qism K^{\bullet}(\xx-\UU;X[\UU^{-1}]),
	\]
	where $K^{\bullet}(\xx-\UU;X[\UU^{-1}])$ denotes the Koszul co-complex 
	of the system of elements $\xx-\UU$ in $R[\UU]$ with respect to the 
	complex of inverse polynomials $X[\UU^{-1}]$.
\end{theorem} 

As mentioned above, if $\xx$ is a weakly pro-regular sequence, then $\check{C}_{\xx}(X)$ is a representative of ${\rm R} \Gamma_{\mathfrak{a}}(X)$ 
in the derived category. Therefore we get the following application:

\begin{corollary} \label{int-4}
	Assume that $\xx$ is a weakly pro-regular sequence with $\mathfrak{a} =\xx R$. 
	Then there are isomorphisms 
	\[
	H^i_{\mathfrak{a}}(X) \cong H^i(\xx-\UU;X[\UU^{-1}]),
	\]
	for all $i \in \mathbb{Z}$, where $H^i(\xx-\UU;X[\UU^{-1}])$ denotes the 
	$i$-th Koszul co-homology.
\end{corollary}

See \ref{coh-6} for the proof. That is, for a Noetherian ring $R$ we are 
able to compute the left derived 
functors of the completion resp. the local cohomology functors for any 
ideal in terms of Koszul homology resp. Koszul co-homology. As a further application we prove a certain duality: 

\begin{theorem} \label{int-5}
	Let $\xx$ denote a system of elements of a commutative ring $R$. Then there 
	is an isomorphism 
	\[
	\Hom_R(K^{\bullet}(\xx-\UU;X[\UU^{-1}]),Y) \cong  
	K_{\bullet}(\xx-\UU;\Hom_R(X,Y)[[\UU]])
	\]
	for two $R$-complexes $X,Y$. 
\end{theorem}

The proof is shown in  \ref{dual-6}. There are several application of this 
duality (see Section 7 for the details). In Section 2 we collect various 
homological preliminaries needed in the paper. Section 3 is devoted to the 
free resolutions of the \v{C}ech complex $\check{C}_{\xx}$ for a system of 
elements $\xx = x_1,\ldots, x_r$. In Section 4 there is a description of the 
\v{C}ech homology as Koszul homology of the complex of formal power series. 
In Section 5 we add some further properties to weakly pro-regular and 
$M$-weakly pro-regular sequences as they were studied in \cite{SS}. 
The module of inverse polynomials is introduced in Section 6. It is 
used in order to describe \v{C}ech cohomology as Koszul co-homology 
of certain complexes of inverse polynomials. Duality is studied in Section 7. 
It provides a certain duality similar to the local duality for Gorenstein 
rings (see \ref{dual-8}). In Section 8 we discuss the behaviour of the 
(co-)homologies by enlarging the number of elements in the corresponding systems.

The results of the present paper are extensions of some of those of \cite{SS} for 
\v{C}ech homology. The results about \v{C}ech cohomology are new and use Macaulay's 
inverse systems. For the notations we follow those of \cite{SS}.

\section{Preliminaries}

\begin{z} {\it Homological peliminaries.} \label{prel-1}
	(A) Let $X,Y$ two complexes of $R$-module and $\phi : X \to Y$ a morphism. 
	We use the notion of the cone $C(\phi)$ and the fibre $F(\phi)$ 
	as defined in \cite[Section 1.5]{SS}. Note that 
	$F(\phi)=C(\phi)^{[-1]}$. Moreover, there are short exact sequences 
	of $R$-complexes 
	\[
	0 \to Y \to C(\phi) \to X^ {[1]} \to 0 \quad {\rm and} \quad 0 \to Y^ {[-1]}
	\to F(\phi) \to X \to 0.
	\]
	Note that $\phi : X \to Y$ is a quasi-isomorphism if and only if
	$C(\phi)$ (respectively $F(\phi)$) is exact.\\
	(B) For further details about homological algebra we refer 
	to \cite{SS}. Here we recall the following:  A morphism $f : X\to Y$ of complexes is called a quasi-isomorphism if it
	induces isomorphisms at the (co-) homology level.  We use  the notation
	$X \qism Y$ or $Y\lqism X$ to denote a quasi-isomorphism. As usual we say
	that two complexes $X$ and $Y$ are \textit{quasi-isomorphic} and we
	write $X\simeq Y$ if there is a finite sequence of quasi-isomorphisms
	$X\qism \cdots \lqism  Y$.
\end{z}

\begin{z} {\it Kozul complexes I.} \label{prel-2}
	(A) Let $X$ denote a complex of $R$-modules and $x$  an
	element of $R$. We have a natural morphism $\mu(x;X): X \stackrel{x}{\to} X$
	induced by the multiplication map $X^i \stackrel{x}{\to} X^i$
	for all $i \in \mathbb{Z}$.
	
	We define the ascending and descending Koszul complexes by
	\[
	K^{\bullet}(x;X) = F(\mu(x;X)) \text{ and } K_{\bullet}(x;X) =
	C(\mu(x;X)).
	\]
	For a sequence $\xx = x_1,\ldots,x_k$ of elements of $R$ and an element
	$y  \in R$ we denote by $\xx,y$ the sequence $x_1,\ldots,x_k, y$. Then
	we define inductively
	\[
	K^{\bullet}(\xx,y;X) = K^{\bullet}(y;K^{\bullet}(\xx;X))) \text{ and }
	K_{\bullet}(\xx,y;X) = K_{\bullet}(y;K_{\bullet}(\xx;X))).
	\]
	In particular, we write $K_{\bullet}(\xx)$ respectively $K^{\bullet}(\xx)$ in the case when
	$X = R$ considered as a complex in degree zero.\\
	(B) Let $m \geq n$ be positive integers. For $x \in R$ and an $R$-complex 
	$X$ there are the following two commutative diagrams 
	\[
	\xymatrix{
		X \ar[r]^{x^{m-n}} \ar[d]^{x^m} & X \ar[d]^{x^n}\\
		X \ar@2{-}[r] &X
		}
	\mbox{ and }
	\xymatrix{
		X  \ar@2{-}[r] \ar[d]^{x^n} & X \ar[d]^{x^m} \\
		X \ar[r]^{x^{m-n}} &X.}
	\]
	So there are natural morphisms $K_{\bullet}(x^m;X) \to K_{\bullet}(x^n;X)$ and 
	$K^{\bullet}(x^n;X) \to K^{\bullet}(x^m;X)$ for all $m \geq n$. \\
	(C) Let $\xx = x_1,\ldots,x_r$ denote a system of elements of $R$. 
	We set $\xx^{(n)} = x_1^n,\ldots,x_r^n$ for an integer $n$. Then the previous construction extends to natural morphisms 
	\[
	K_{\bullet}(\xx^{(m)};X) \to K_{\bullet}(x^{(n)};X) \mbox{ and }  
	K^{\bullet}(x^{(n)};X) \to K^{\bullet}(x^{(m)};X) 
	\]
	for all integers $m \geq n$.
\end{z}

\begin{z} {\it Koszul complexes II.} \label{prel-3}
	(A) A more general definition of the Koszul complex is the following 
	(see e.g. \cite[Section 1.6]{BrH}). Let $f : L \to R$ denote 
	an $R$-homomorphism. Then $K_{\bullet}(f)$ is defined as the exterior algebra 
	$\wedge L$ with the differential induced by $f$. For an $R$-module $M$ we define $K_{\bullet}(f;M) = K_{\bullet}(f) \otimes_R M$ and $K^{\bullet}(f;M) = \Hom_R(K_{\bullet}(f),M)$ 
	(see \cite{BrH}). For a system of elements $\xx = x_1,\ldots,x_r$ we define 
	$K_{\bullet}(\xx;M) = K_{\bullet}(f;M)$, where $f : R^r \to R$ is the 
	$R$-linear map defined by $e_i \mapsto x_i, i = 1,\ldots,r$, where 
	$e_1,\ldots,e_r$ is the standard free basis of $R^r$. Recall that this definition 
	of $K_{\bullet}(\xx;M)$ and $K^{\bullet}(\xx;M)$ coincides 
	with those of \ref{prel-3}\\
	(B) If $f_i : L_i \to R, i = 1,2,$ are two $R$-linear maps and 
	$\phi : L_1 \to L_2$ is an $R$-homomorphism then there is a morphism 
	of $R$-complexes $K_{\bullet}(f_1;M) \to K_{\bullet}(f_2 \cdot \phi;M)$. 
	This is an isomorphism of $R$-complexes if $\phi$ is an isomorphism. In particular, let $\underline{y} = y_1,\ldots,y_r$ and $\xx = x_1,\ldots,x_r$ denote two sequences of elements such that $(\underline{y}) = (\xx)\cdot\mathfrak{A}$ for an invertible $r\times r$-matrix $\mathfrak{A}$. 
	Then there are isomorphisms $K_{\bullet}(\underline{y};M) \cong K_{\bullet}(\xx;M)$ and $K^{\bullet}(\underline{y};M) \cong K^{\bullet}(\xx;M)$.
	\\
	(C) Let $\xx = x_1,\ldots,x_r$ denote a system of elements in $R$. 
	For a positive integer $n$ we write $\xx^{(n)}  = x_1^n,\ldots,x_r^n$ as above.
	For two integers $m \geq n$ there is a commutative diagram 
	\[
	\xymatrix{
		R^r \ar[r]^{\xx^{(m)}} \ar[d]^{\mathfrak{B}} & R \ar@2{-}[d] \\
		R^r \ar[r]^{\xx^{(n)}}  &R,
	}
	\]
	where $\mathfrak{B}$ is the diagonal matrix with entries 
	$x_i^{m-n}, i = 1,\ldots,r$.
	Whence there is a morphism of complexes $K_{\bullet}(\xx^m;M) \to K_{\bullet}(\xx^n;M)$. It coincides with the iterative construction as done above.\\
	(D) All of the previous constructions extend to $K_{\bullet}(\xx;X)$ and 
	to $K^{\bullet}(\xx;X)$ for an $R$-complex $X$.
\end{z}

For some induction arguments we need the following technical lemma.

\begin{proposition} \label{prel-7}
	Let $X$ denote an $R$-complex and $x \in R$ an element. Then there is 
	a short exact sequence of $R$-complexes
	\begin{itemize}
		\item[(a)] 
		$0 \to (0:_X x)_{[-1]} \to K_{\bullet}(x;X) \to X/xX \to 0$,
		\item[(b)] 
		$0  \to 0:_X x \to K^{\bullet}(x;X) \to (X/xX)^{[+1]} \to 0.$
	\end{itemize}
	Here $0:_X x = \Hom_R(R/xR,X)$ and $X/xX = X \otimes_R R/xR$ denote the 
	sub-complex and the factor-complex respectively.
\end{proposition}

\begin{proof}
	The proof of (a) is clear by the definition of $K_{\bullet}(x;X)$ as the 
	cone of $\mu(x;X): X \stackrel{x}{\to} X$. Note that the 
	maps are morphisms of $R$-complexes. The proof of (b) is similar by using the 
	fibre. 
\end{proof}

\begin{z} \label{prel-4} {\it Telescope and Microscope.} (see \cite{GM} and 
	\cite{SS})
	(A) Let $\{X_n, \phi_n\}$ denote a direct system of $R$-complexes with 
	$\phi_n : X_n \to X_{n+1}$ the transition map. Then there is a natural map 
	\[
	\Phi_X : \oplus X_n \to \oplus X_n, \; x_n \mapsto x_n -\phi_n(x).
	\]
	The telescope $\Tel(\{X_n\})$ is defined as the cone $C(\Phi_X)$. 
	Whence there is a natural quasi-isomorphism 
	\[
	\Tel(\{X_n\}) \to \varinjlim X_n.
	\]
	Let $X$ denote an $R$-complex. 
	For a system of elements $\xx$ the Koszul co-complexes  $K^{\bullet}(\xx^{(n)};X)$ with the natural maps forms a direct system.
	We define $T(\xx;X) = \Tel(\{K^{\bullet}(\xx^{(n)};X)\})$.\\
	(B) Let $\{X_n, \psi_n\}$ denote an inverse system of $R$-complexes with 
	transition maps $\psi_n : X_{n+1} \to X_n$.  Then there is a natural map 
	\[
	\Psi_X : \textstyle{\prod} X_n \to \textstyle{\prod}X_n, \; (x_n) 
	\mapsto (x_n - \psi_n(x_{n+1})).
	\]
	The microscope $\Mic(\{X_n\})$ is defined as the fibre $F(\Psi_X)$.
	So there is a natural morphism 
	\[
	\varinjlim X_n \to\Mic(\{X_n\}). 
	\]
	It is a quasi-isomorphism provided $\{X_n\}$ satisfies degree-wise the 
	Mittag-Leffler condition.
	For an $R$-complex $X$ and a system of elements $\xx$ the Koszul complexes $K_{\bullet}(\xx^{(n)};X)$ with the natural maps forms an inverse system.
	We define $M(\xx;X) = \Mic(\{K_{\bullet}(\xx^{(n)};X)\})$.\\
	(C) For the homology of $\Mic(\{X_n\})$ there are the following short exact sequences 
	\[
	0 \to \varprojlim{}^1 H_{i+1}(X_n) \to H_i(\Mic(\{X_n\})) \to 
	\varprojlim H_i(X_n) \to 0
	\]
	for all $i \in \mathbb{Z}$. For the proof in the case of modules we 
	refer to \cite{GM} and to \cite[4.2.3]{SS} for the general case. In particular, if $\{X_n\} \to \{Y_n\}$ is a family of inverse systems of 
	quasi-isomorphisms it yields a quasi-isomorphism $\Mic(\{X_n\}) \to \Mic(\{Y_n\})$.
\end{z}

\begin{z} {\it Weakly pro-regular sequences.}\label{prel-5}
	(A) Let $M$ denote an $R$-module. The sequence $\xx$ is called 
	$M$-weakly pro-regular if the inverse system of Koszul homology 
	modules $\{H_i(\xx^{(n)};M)\}$ is pro-zero for all $i \not= 0$ with 
	$\xx^{(n)} = x_1^n, \ldots,x^n$. Note that the system  $\{H_i(\xx^{(n)};M)\}$ 
	is pro-zero if for all $n$ there is an $m \geq m$ such that the natural 
	map $H_i(\xx^{(m)};M) \to H_i(\xx^{(n)};M)$ is zero. The sequence $\xx = x_1,\ldots,x_r$ is called weakly pro-regular if it is $R$-weakly pro-regular 
	(see \cite{Sp2} or \cite[Section 7.3]{SS} for a more detailed description).\\
	(B) In case of a single element $x \in R$ and an $R$-module $M$ the element 
	$x$ is called $M$-weakly pro-regular if for $n \geq 1$ there is an 
	$m \geq n$ such that $H_1(x^m;M) \stackrel{x^{m-n}}{\longrightarrow} 
	H_1(x^n;M)$ is the zero map. By identifying $H_1(x^n;M) = 0:_M x^n$ it follows that this is equivalent to the fact that $M$ is of bounded 
	$xR$-torsion (see e.g. \cite{Sp2} or \cite{SS} for more details).
\end{z}

\begin{z} {\it Notations.} \label{prel-6}
	(A) Let $M$ be an $R$-module and let $\Lambda$ be a set. Then $M^{(\Lambda)}$ 
	and $M^{\Lambda}$ denote the direct sum and the direct product of $M$ over $\Lambda$ respectively.\\
	(B) Let $\mathbb{N}$ denote the non-negative integers; and  let
	$\mathbb{N}_+$ denote the positive integers. Let $R$ denote a commutative ring. Let $U$ resp. $\UU = U_1,\ldots,U_r$ 
	denote a variable resp. a system of $r$ variables over $R$. Then $R[U], R[\UU]$ and $R[[U]], R[[\UU]]$ denote the corresponding polynomial and 
	formal power series rings over $R$. Note that $R[\UU] \cong R^{(\mathbb{N}^r)}$ is a free $R$-module.  We identify $R[[\UU]] = \Hom_R(R[\UU],R)$ and therefore $R[[\UU]] \cong R^{\mathbb{N}^r}$. \\
	(C) For an $R$-module $M$ we consider $M[\UU]$ and $M[[\UU]]$ the polynomial 
	and the formal power series module over $\UU$. Clearly $M[\UU]$ and $M[[\UU]]$ possesses the structure of an $R[\UU]$-module. Furthermore,  
	$M[[\UU]]$ admits the structure of an $R[[\UU]]$-module and $M[[\UU]] \cong  
	\Hom_R(R[\UU],M)$. Moreover, there is an isomorphism 
	$M[[\UU]] \cong \varinjlim M[\UU]/\UU^{(n)}M[\UU]$, where $\UU^{(n)} = U_1^n,\ldots,U_r^n$. \\
	(D) Let $X$ denote an $R$-complex. Similarly we also have the degree-wise
	identification of the $R[\UU]$-complex $\Hom_R(R[\UU], X)$ with
	$X[[\UU]] := X[[U_1,\ldots,U_r]]$. Moreover, $X[[\UU]]$ is an $R[\UU]$-complex and
	an $R[[\UU]]$-complex hence also an $R$-complex. There is an obvious
	isomorphism of $R$-complexes $X[[\UU]] \cong X^{\mathbb{N}^k}$. 
	The construction is functorial in $X$, a morphism of $R$-complexes
	$X \to Y$  induces a morphism $X[[\UU]] \to Y[[\UU]]$ of $R[[\UU]]$-complexes.
	Moreover, if $X\to Y$ is a quasi-isomorphism, then $X[[\UU]] \to Y[[\UU]]$ is a quasi-isomorphism too. This
	follows easily since  direct products commute with homology.\\	
\end{z}

\section{Resolutions of  \v{C}ech complexes}
\begin{z} {\it The \v{C}ech complex.} \label{comp-1}
	(A) Let $R$ denote a commutative ring and $x \in R$ an element. Let 
	$M$ be an $R$-module. Then $M_x$ denotes the localization with respect to
	$S = \{x^n | n \in \mathbb{N}\}$, the multiplicatively 
	closed set by the non-negative powers of $x$. There is the 
	natural homomorphism $\iota = \iota(x;M) : M \to M_x, m \mapsto m/1$.
	
	Let $X$ denote an $R$-complex. Then the previous construction 
	extents to a morphism $\iota(x;X): X \to X_x \cong X \otimes_R R_x$. 
	We define $\check{C}_x(M) = F(\iota(x;X))$ as the fibre of $\iota(x;X)$. 
	With $\check{C}_x = \check{C}_x(R)$ it follows by the definition 
	$\check{C}_x(X) \cong \check{C}_x \otimes_R X$. \\
	(B) Let $\xx = x_1,\ldots,x_r$ be a system of elements of $R$ and $y \in R$ 
	an element. Then we define $\check{C}_{\xx,y}(X) = \check{C}_y(\check{C}_{\xx}(X))$. In particular it follows that 
	\[
	\check{C}_{\xx}(X) \cong \check{C}_{x_1} \otimes_R \cdots \otimes_R \check{C}_{x_r} 
	\otimes_RX.
	\]
	(C) Moreover, it is known that $R_x = \varinjlim \{R_n,x\}$ with 
	$R_n = R$ and $R_n \to R_{n+1}$ is the multiplication by $x$. For two integers $m \geq n$ there is a commutative diagram 
	\[
	\xymatrix{
	K^{\bullet}(x^n): \ar[d] &	0 \ar[r]  & R \ar[r]^{x^n} \ar@2{-}[d] & R \ar[r] \ar[d]^{x^{m-n}} & 0 \\
	K^{\bullet}(x^m): &	0 \ar[r] & R \ar[r]^{x^m} & R \ar[r] & 0.
	}
	\]
	With the previous arguments it yields that $\varinjlim K^{\bullet}(x^n)	 \cong \check{C}_x$. As an easy application it follows that $\varinjlim K^{\bullet}(x^n;X) \cong \check{C}_{x}(X)$ for an $R$-complex $X$. 
	
	Let $\xx = x_1,\ldots,x_r$ denote a sequence of elements of $R$. Then the previous arguments imply the isomorphism $\varinjlim K^{\bullet}(\xx^{(n)};X) \cong \check{C}_{\xx}(X)$ for an $R$-complex $X$. 
\end{z}

\begin{z} {\it First free resolution.}\label{comp-2}
	By the construction $\check{C}_{\xx} \cong \otimes_{i=1}^r \check{C}_{x_i}$. Therefore 
	it has the structure 
	\[
	\check{C}_{\xx} : 0 \to R \to \oplus_{i} R_{x_i} \to \oplus_{i < j} R_{x_ix_j} \to \ldots \to R_{x_1\cdots x_r} \to 0
	\]
	for a system of elements $\xx = x_1,\ldots,x_r$. That is, $\check{C}_{\xx}$ is a bounded 
	complex of flat $R$-modules, whence of finite flat dimension.
	
	With the previous notation we have $\check{C}_{\xx}(X) \cong \varinjlim 
	K^{\bullet}(\xx^{(n)};X)$. By the definition of the direct limit there 
	is a short exact sequence of complexes 
	\[
	0 \to \oplus_{n \in \mathbb{N}} K^{\bullet}(\xx^{(n)};X) \stackrel{\Phi}{\longrightarrow} \oplus_{n \in \mathbb{N}} K^{\bullet}(\xx^{(n)};X) 
	\to \check{C}_{\xx}(X) \to 0.
	\]
	We define $T(\xx;X)$ the telescope $\Tel(\{K^{\bullet}(\xx^{(n)};X)\})$ as above. By the construction of the cone 
	and the above short exact sequence there is a quasi-isomorphism of complexes
	\[
	T(\xx;M) \quism \check{C}_{\xx}(X).
	\]
	Now it follows that $T(\xx) = T(\xx;R)$ is a complex of free $R$-modules since 
	each $K^{\bullet}(\xx^{(n)};R)$ is a complex of free $R$-modules and $T(\xx)$ 
	is formed by a direct sum of them. Hence it follows that $T(\xx)^i = 0$ for 
	$i \notin [-1,r]$. That means $T(\xx)$ is a bounded free resolution of $\check{C}_{\xx}$ 
	for a sequence of elements $\xx = x_1,\ldots,x_r$ of $R$.
		
	This construction is a slight modification of a corresponding one for modules 
	as done by Greenlees and May (see \cite{GM}).
\end{z}

\begin{z} {\it Second free resolution.} \label{comp-3}
	(A) Let $R[U]$ denote the polynomial ring in the variable $U$ over $R$. 
	The short exact sequence $0 \to R[U] \stackrel{1-xU}{\longrightarrow} R[U] \to R_x \to 0$ provides a free resolution of $R_x$ as an $R$-module. This follows since 
	$R_x = \varinjlim \{R_n,x\}$ with 
	$R_n = R$ and $R_n \to R_{n+1}$ is the multiplication by $x$.
	Let 
	\[
	P_x : 0 \to R[U]  \stackrel{1-xU}{\longrightarrow} R[U] \to 0 
	\]
	denote this free resolution of $R_x$. There is a natural morphism 
	$g_x: R\to P_x$ inserted in the commutative diagram
	\[
	\xymatrix{
		R \ar@2{-}[r] \ar[d]_{g_x} & R  \ar[d]^{\iota_x} \\
		P_x  \ar[r]^-\sim & R_x 
	}
	\]
	Let $\check{L}_x:= F(g_x)$ denote the fibre of the natural morphism 
	$g_x: R \to P_x$. That is 
	\[
	\ldots \to 0 \to RU^{-1} \oplus R[U] \stackrel{\phi}{\longrightarrow} R[U] 
	\to 0 \to \ldots, \; \phi : 
	(rU^{-1}, f(U)) \mapsto r - (1-xU) f(U).
	\]
	Then there is a natural quasi-isomorphism 
	\[
	F(g_x)= \check{L}_x \quism \check{C}_x=F(\iota_x).
	\] 
	It provides a $R$-free resolution of $\check{C}_x$ (see also \cite{Sp2} or \cite{SS} for more details). \\
	(B) Let $X$ denote an $R$-complex. Then we define $\check{L}_x(X) = \check{L}_x \otimes_RX$. Also note that $\check{L}_x(X) \cong F(g_x \otimes \id_X)$. Because $\check{L}_x \to \check{C}_x$ is a quasi-isomorphism between 
	bounded complexes of flat $R$-modules it induces a quasi-isomorphism 
	$\check{L}_x(X) \to \check{C}_x(X)$. Moreover, if $X \to Y$ is a quasi-isomorphism 
	of $R$-complexes it induces quasi-isomorphisms $\check{L}_x(X) \to \check{L}_x(Y)$ and  $\check{C}_x(X) \to \check{C}_x(Y)$. 
	
	For a system of elements $\xx = x_1,\ldots, x_r$ and $y \in R$ we define recursively 
	$\check{L}_{\xx,y}(X) = \check{L}_y(\check{L}_{\xx}(X))$. It follows that there is a 
	natural quasi-isomorphism $\check{L}_{\xx}(X) \to \check{C}_{\xx}(X)$. Moreover, 
	by the definition there is an isomorphism $\check{L}_{\xx}(X) \cong \check{L}_{x_1}\otimes_R \cdots \otimes_R \check{L}_{x_r} \otimes_R X$. \\
	(C) Let $\xx = x_1,\ldots,x_r$ denote a system of elements. Then it follows that 
	$\check{L}_{\xx}$ is a bounded complex of free $R$-modules. Because of the 
	quasi-isomorphism $\check{L}_{\xx} \to \check{C}_{\xx}$ it is a bounded free resolution 
	of the  \v{C}ech complex $\check{C}_{\xx}$. 
	We have that $\check{L}_x^i = 0$ for $i \notin [0,r]$. 
\end{z}

\begin{z} \label{comp-4} {\it Some constructions.}
	As before let $R[U]$ denote the polynomial ring in the variable $U$ over 
	$R$. As above we identify $R[U] = \oplus_{n \in \mathbb{N}} R{e_n}$, where 
	$e_n = (0, \ldots,0,1,0,\ldots)$ with $1$ at the $n$-th place, $n \in \mathbb{N}$. 
	Here $\mathbb{N}$ denotes the non-negative integers, while $\mathbb{N}_+$ 
	denotes the positive integers. Then we define a complex 
	\[
	\mathcal{L}_x : \ldots \to 0 \to R[U] \stackrel{\psi}{\longrightarrow} 
	U R[U] \to 0 \to \ldots,
	\]
	where $UR[U] = \oplus_{n \in \mathbb{N}_+} R{e_n}$ and the $R$-homomorphism 
	$\psi$ is given by 
	\[
	\psi : R[U] \to UR[U], r(U) \mapsto r(0) -(1-xU)r(U)
	\] 
	with $r(U) = r_0+r_1U + \dots + r_nU^n \in R[U]$. Then we consider the 
	following morphism of $R$-complexes
	\[
	\xymatrix{
		\mathcal{L}_x : \ar[d] & 0 \ar[r] & R[U] \ar[r]^{\psi} \ar[d]^{f_0} & 
		UR[U] \ar[r] \ar[d]^{f_1} & 0 \\
		\check{C}_x : & 0 \ar[r] & R \ar[r]^{\iota_x} & R_x \ar[r] & 0
	}
	\] 
	with 
	\[
	f_0: R[U] \to R, r(U) \mapsto r(0),\quad f_1: UR[U] \to R_x, r(U) \mapsto r(1/x).
	\] 
	Clearly $\iota_x \cdot f_0 = f_1 \cdot \psi,$ so that it is a morphism of $R$-complexes.
\end{z}

\begin{lemma} \label{comp-5}
	With the previous notation the morphism $\mathcal{L}_x \to \check{C}_x$ is a quasi-isomorphism.
\end{lemma}

\begin{proof}
	It will be enough to show that the induced $R$-homomorphisms 
	\[
	g_0 : \Ker \psi \to \Ker \iota_x \mbox{ and } 
	g_1 : \Coker \psi \to R_x/\iota_x(R)
	\] 
	are isomorphisms. First let $r \in \Ker \iota_x \setminus \{0\}$, that 
	is $r \in 0:_R x^n$ for a certain integer $n \geq 1$. Then we define 
	$r(U) = r_0+r_1U +\ldots +r_{n-1}U^{n-1}$ with $r_i = x^ir$ for $i = 0,\ldots,n-1$. Then $r(U) \in \Ker \psi$, so that $g_0$ is onto. 
	
	Let $r(U) = r_0 + r_1U + \ldots+r_nU^n \in \Ker g_0$. It yields that 
	$r_0 = 0$ and $r(U) \in \Ker \psi$. Therefore $xr_n = 0$ and $r_i = xr_{i-1}$ for $i = 1,\ldots,n$ and $r(U) = 0$, whence $g_0$ is injective. 
	
	Let $r/x^n +\iota_x(R) \in \Coker \iota_x$ with $n \geq 1$. It is the image 
	of $rU^n +\im \psi$, and $g_1$ is onto. Now let $r(U) + \im \psi \in 
	\Ker g_1$ with $r(U) = r_1U + \ldots+r_nU^n$ with $n \geq 1$. Then 
	$r_1/x +\ldots + r_n/x^n = -r_0/1$ for a certain $r_0 \in R$. This provides
	an element $s \in R$ and a certain $k$ such that 
	$r_n + xr_{n-1}+ \ldots+x^nr_0 = s \in 0:_R x^k$. Now we define 
	$s(U) = s_0 + s_1U + \ldots +s_{n-1}U^{n-1}$ with 
	\[
	s_i = -\sum_{j=0}^{i}x^{i-j}r_j \quad \mbox{ for } i = 0,\ldots,n-1.
	\]  
	Therefore $r(U)-[s(0) - (1-xU)s(U)] = (r_n+xr_{n-1} + \ldots + x^{n-1}r_1 +x^nr_0)U^n = sU^n$ with $s \in 0:_Rx^k$. In order to prove that $g_1$ 
	is injective it will be enough to show that $sU^n \in \im \psi$. But this 
	follows since $sU^n \cong xsU^{n+1} \mod (\im \psi), n \geq 1,$ and by 
	induction $sU^n \cong x^ksU^{n+k} \mod (\im \psi)$ which finishes the proof 
	since $x^k s = 0$.
\end{proof}

In the following we shall give a different proof that $\check{L}_x$ is a free resolution 
of $\check{C}_x$. Moreover, we study the relation of the two free resolutions 
$\check{L}_x$ and $\mathcal{L}_x$ of $\check{C}_x$. 

\begin{theorem} \label{comp-6}
	Let $x \in R$ denote an element of a commutative ring $R$. With $\phi$ as in 
	\ref{comp-3} (A) and $\psi$ as in \ref{comp-4} 
	there is a commutative diagram with exact columns
	\[
	\xymatrix{
		0 \ar[d] &  & 0 \ar[d]  &  0 \ar[d]& \\
		\mathcal{E}: \ar[d]&0 \ar[r] & RU^{-1} \ar[r]^{\tau} \ar[d]^{i_1} & R \ar[d]^{i_0} \ar[r] & 0 \\ 
		\check{L}_x: \ar[d]& 0 \ar[r] & RU^{-1} \oplus R[U] \ar[d]^{p_1} \ar[r]^{\phi} & 
		R[U] \ar[r] \ar[d]^{p_0} &0\\
		\mathcal{L}_x: \ar[d] & 0 \ar[r] & R[U] \ar[r]^{\psi} \ar[d] & 
		UR[U] \ar[r] \ar[d] & 0\\
		0&  & 0   &  0. &
	}
	\]
	Here we have the injections 
	\[
	i_1 : RU^{-1} \to RU^{-1} \oplus R[U], rU^{-1} \mapsto (rU^{-1},0) 
	\mbox{ and } i_0 : R \to R[U], r \mapsto r
	\] 
	and the projections 
	\[p_1 : RU^{-1} \oplus R[U] \to R[U], (rU^{-1},r(U)) \mapsto r(U) \quad \mbox{ and } \quad 
	p_0 : R[U] \to UR[U], r(U) \mapsto r(U)-r(0)
	\] 
	and $\tau$ is $\operatorname{id}_R$. 
	The natural morphism $\check{L}_x \to \mathcal{L}_x$ is a quasi-isomorphism. 
	Moreover $\mathcal{L}_x$ is a direct summand of $\check{L}_x$. 
\end{theorem}

\begin{proof}
	It is easy to check that $i_0 \cdot \tau = \phi \cdot i_1$ and 
	$p_0 \cdot \phi = \psi \cdot p_1$ and that the vertical sequences 
	are short exact sequences. So there is a morphism 
	$\check{L}_x \to \mathcal{L}_x$. By the long exact cohomology sequence 
	it follows that it is a quasi-isomorphism since $\mathcal{E}$ is exact. 
	Clearly $\mathcal{L}_x$ is a direct summand of $\check{L}_x$.
%	For the final statement we define 
%	$$
%	\pi_0 : R[U] \to R, r(U) \mapsto r(0) \; \mbox{ and } \;
%	\pi_1 : RU^{-1} \oplus R[U] \to RU^{-1}, (rU^{-1}, r(U)) \mapsto (r-r(0))U^{-1}.
%	$$ 
%	Then $\tau \cdot \pi_1  = \pi_0 \cdot \phi$ and 
%	$\Pi: \check{L}_x \to \mathcal{E}$ is a morphism of $R$-complexes. Moreover 
%	$\pi_n \cdot i_n = \id$ for $n = 0,1$ so that the morphism is split.  
\end{proof}

\begin{z} \label{comp-7}
	(A) Let $\mathcal{L}_x$ denote the $R$-complex as defined in \ref{comp-4}. 
	Note that it is the fibre of $\psi : R[U] \to U R[U]$, that is, $\mathcal{L}_x = F(\psi)$.
	It is a complex of free $R$-modules. For an $R$-complex $X$ we put $\mathcal{L}_x(X) = \mathcal{L}_x \otimes_R X$. Therefore, if $X \to Y$ is a quasi-isomorphism, then 
	$\mathcal{L}_x(X) \to \mathcal{L}_x(Y)$ is a quasi-isomorphism too. 
	By the construction there are quasi-isomorphisms 
	\[
	\check{L}_x(X) \to \mathcal{L}_x(X) \to \check{C}_x(X) .
	\]
	Note that $\check{L}_x, \mathcal{L}_x, \check{C}_x$ are bounded complexes of flat $R$-modules. \\
	(B) Let $\xx = x_1,\ldots,x_r$ be a system of elements of $R$ and let $y\in R$. 
	Then we define recursively $\mathcal{L}_{\xx,y}(X) = \mathcal{L}_y(\mathcal{L}_{\xx}(X))$. By induction and the use of (A) it follows 
	that there are quasi-isomorphisms
	\[
	\check{L}_{\xx}(X) \to \mathcal{L}_{\xx}(X) \to \check{C}_{\xx}(X) 
	\]
	for an $R$-complex $X$.\\
	(C) ({\textit{Third free resolution.}}) For $X = R$ the previous 
	quasi-isomorphisms provide an $R$-free resolutions of 
	the \v{C}ech complex $\check{C}_{\xx}$, that is $\mathcal{L}_{\xx} \to \check{C}_{\xx}$. 
	Note that $\check{L}_{\xx}$ is not a minimal free resolution. 
\end{z}

\section{\v{C}ech homology via Koszul complexes}
For a system of elements $\xx = x_1,\ldots,x_r$ of a commutative ring $R$ 
we have that the quasi-isomorphism $\check{L}_{\xx} \to \mathcal{L}_{\xx}$ provides 
free resolutions of the \v{C}ech complex $\check{C}_{\xx}$. For an $R$-complex $X$ 
we denote the homology of $\Hom_R(\mathcal{L}_{\xx},X) \quism \Hom_R(\check{L}_{\xx},X)$ 
as the \v{C}ech homology of $X$. Recall that it is not possible to work with 
$\Hom_R(\check{C}_{\xx},\cdot)$ since it is not functorial in the derived category. 
For that reason it is important to have free resolutions of $\check{C}_{\xx}$. 
This part is an improvement of results of \cite[Chapter 8]{SS} needed in the subsequent parts.

\begin{z} \label{hoc-1} {\it The complex $\Hom_R(\mathcal{L}_{\xx},X)$.}
	(A) We start with the complex $\mathcal{L}_x$ with respect to a single element 
	$x \in R$. That is 
	$$
	\mathcal{L}_x : 0 \to R[U] \stackrel{\psi}{\longrightarrow} UR[U] \to 0.
	$$
	For an $R$-module $M$ we apply the functor $\Hom_R(\cdot,M)$. Therefore 
	we get the complex
	\[
	\Hom_R(\mathcal{L}_x,M) : 0 \to \Hom_R(UR[U],M) \stackrel{\psi^{\vee}}{\longrightarrow} \Hom_R(R[U],M) 
	\to 0.
	\]
	The multiplication $R[U] \stackrel{U}{\longrightarrow} UR[U]$  induces an isomorphism $\Hom_R(UR[U],M) \to \Hom_R(R[U],M)$. 	
	For an $R$-module $M$ we identify $\Hom_R(R[U],M)$ by $M[[U]]$, the formal power series 
	module in the variable $U$ over $M$. That is, $M[[U]]$ consists of all 
	elements $\sum_{i\geq 0}m_i U^i$ with $m_i \in M, i \in \mathbb{N}$. 
	Then the complex becomes the following form
	\[
	\Hom_R(\mathcal{L}_x,M) : 0 \to M[[U]] \stackrel{\rho}{\longrightarrow}
	M[[U]] \to 0
	\]
	where the homomorphism $\rho$ is given by 
	\[
	\rho: M[[U]] \to M[[U]], \; m(U) \mapsto xm_0 + (xm_1-m_0)U + (xm_2-m_1)U^2 
	+ \ldots,
	\]
	where $m(U) = m_0 + m_1U +m_2U^2 + \ldots$. It follows that 
	$$
	\rho(m(t)) = (x-U) m(U) = \textstyle{\sum}_{i\geq 0}(xm_i-m_{i-1})U^i \mbox{ with } m_{-1}=0.
	$$ 
	Therefore the $R$-homomorphism can be considered as an $R[U]$-homomorphism, 
	namely multiplication by $x-U$. We may consider the complex as a complex 
	of $R[U]$-modules and it becomes 
	\[
	K_{\bullet}(x-U;M[[U]]): 0 \to M[[U]] \stackrel{x-U}{\longrightarrow} 
	M[[U]] \to 0,
	\]
	the Koszul complex of the $R[U]$-module $M[[U]]$ with respect to $x-U$. 
	Since $x-U$ annihilates the homology of $\mathcal{K}(x;M)$ they are $R$-modules 
	because of $R[U]/(x-U)R[U] \cong R$. 
	
	This construction extents in a natural way 
	to an $R$-complex $X$. This follows by a degree-wise inspection. We have 
	\[
	\Hom_R(\mathcal{L}_x,X) : 0 \to \Hom_R(UR[U],X) \stackrel{\psi^{\vee}}{\longrightarrow} \Hom_R(R[U],X) \to 0.
	\]
	Note also that $\Hom_R(\mathcal{L}_x,X) \cong K_{\bullet}(x-U;X[[U]])$ because 
	the Koszul complex is the cone of the multiplication map $\mu_{x-U} : X[[U]] \to X[[U]]$.
	\\
	(B) For a system of elements $\xx = x_1,\ldots,x_r$ of  $R$ and $y \in R$ 
	we define recursively $$\Hom_R(\mathcal{L}_{\xx,y},X) = \Hom_R(\mathcal{L}_y, 
	\mathcal{L}_{\xx}(X)).$$ As a consequence of the recursive definition of the 
	Koszul complex it follows that 
	\[
	\Hom_R(\mathcal{L}_{\xx},X) \cong K_{\bullet}(\xx-\UU;X[[\UU]]),
	\]
	where $\UU = U_1,\ldots,U_r$ denotes a set of variables over $R$ and $\xx-\UU 
	= x_1-U_1,\ldots,x_r-U_r$. That is, the complex for local homology is isomorphic 
	to a certain Koszul complex. Note that $K_{\bullet}(\xx-\UU;X[[\UU]])$ is a complex 
	of $R$- and also of $R[\UU]$-modules. Its homology modules are annihilated 
	by $(\xx-\UU)R[\UU]$; so they are $R[\UU]/(\xx-\UU)R[\UU] \cong R$-modules.
\end{z}

As a first step for understanding the above Koszul complex we shall investigate 
the following particular situation. For an $R$-module $M$ let $\Hom_R(R_x,M) \to 
M$ denote the natural homomorphism $\phi \mapsto \phi(1)$.

\begin{lemma} \label{hoc-2}
	Let $x \in R$ be an element of a commutative ring $R$. 
	Let $M$ be an $R$-module and let $\hat{M}^{xR}$ denote 
	its $xR$-adic completion. 
	\begin{itemize}
		\item[(a)] There is an isomorphism and an epimorphism  
		$$
		H_0(x-U;M[[U]]) \cong M[[U]]/(x-U)M[[U]] \twoheadrightarrow \hat{M}^{xR}.
		$$
		\item[(b)] There is an isomorphism 
		$$H_1(x-U;M[[U]]) \cong \Ker(\Hom_R(R_x,M) \to M).
		$$
		\item[(c)] If $M$ is of bounded $xR$-torsion, then 
		$$H_1(x-U;M[[U]]) = 0 \;\text{ and }\; H_0(K_{\bullet}(x-U;M[[U]])) \cong \hat{M}^{xR}.
		$$ 
	\end{itemize}
\end{lemma}

\begin{proof} 
	By the definition we get 
	\[
	K_{\bullet}(x-U;M[[U]]): \ldots \to 0 \to  M[[U]] \stackrel{x-U}{\longrightarrow} M[[U]] \to 0 \to \ldots
	\]
	such that ${\textstyle \sum}_{i \geq 0} m_iU^i \mapsto {\textstyle \sum}_{i \geq 0}(xm_i-m_{i-1})U^i$ with $m_{-1} =0$. 
	First we investigate the zero-th homology $H_0(x-U;M[[U]])$. Clearly 
	$$
	H_0 := H_0(x-U;M[[U]]) = M[[U]]/(x-U)M[[U]].
	$$ 
	We define a map $f : H_0 \to \hat{M}^{xR} = \varprojlim M/x^i M$, the $xR$-adic completion of $M$ by
	\[
	m(U) + (x-U) M[[U]] \mapsto (m_0+ xm_1 + \ldots + x^{i-1}m_{i-1} + x^iM)_{i \geq 0}
	\]
	where $m(U) = \sum_{i \geq 0} m_iU^i$. The map is well defined since $(x-U)m(U)$ maps to zero and surjective as easily seen.  That is there is a surjective homomorphism
	\[
	H_0(x-U;M[[U]]) = M[[U]]/(x-U)M[[U]] \twoheadrightarrow \hat{M}^{xR}
	\]
	which proves (a). 
	
	For the proof of (b) note that $\Hom_R(R_x,M) \cong \varprojlim \{M,x\}$ 
	with the inverse system $M_{i+1} \stackrel{x}{\longrightarrow} M_i$ with 
	$M_i = M$ for all $i \geq 0$. Whence $\Hom_R(R_x,M)$ is isomorphic 
	to the set sequences $\{(m_0,m_1,\ldots)\} \in M^{\mathbb{N}}$ such that 
	$m_i = xm_{i+1}$ for all $i \geq 0$. Therefore $(m_0,m_1,\ldots) \mapsto 0$ if and only if $m_0+m_1U + \ldots \in \Ker \mu_{x-U}$ which proves the isomorphism in (b).
	
%	For the proof of (b) note that $\Hom_R(R_x,M)$ is isomorphic 
%	to the set sequences $\{(m_0,m_1,\ldots)\} \in M^{\mathbb{N}}$ such that 
%	$m_i = xm_{i+1}$ for all $i \geq 0$. This follows by identification 
%	$\phi(1/x^i) = m_i$ for $i \geq 0$. Under this identification it is easily 
%	seen that $(m_0,m_1,\ldots) \mapsto 0$ if and only if $m_0+m_1U \ldots 
%	\in \Ker \mu_{x-U}$ which proves the isomorphism in (b).
	
	Now suppose that $M$ is of bounded $xR$-torsion. 
	Let $m(U) \in \Ker \mu_{x-U}$. Then $m_{i-1} = xm_i$ for all $i \geq 0$ and $m_{-1} = 0$ therefore 
	$m_i \in 0:_M x^{i+1}$ for all $i \geq 1$ and $m_i \in  \cap_{l \geq 1} x^l M$. 
	Since $M$ is of bounded $x$-torsion there 
	is an integer $k$ such that $0:_M x^k = 0:_M x^{k+j}$ for all $j \geq 0$. Let 
	$m_i = x^k m_{i+k}$ and therefore $x^{k+i+1} m_{i+k} = 0$. That is, $m_{i+k} \in 0:_M x^{i+k} = 0:_M x^k$ 
	and $m_i = 0$. Whence $H_1(x-U;M[[U]]) = 0$.

	Now let $m(U) \in \ker f$. Then there are $n_i \in M$ for all $i \geq 0$ such that 
	$$
	m_0 = -xn_0 \text{ and } m_0+xm_1+ \ldots + x^{i -1}m_{i -1} = -x^in_{i -1} \text{ for all }i \geq 1.
	$$ 
	Whence $x^im_i = x^i n_{i-1} -x^{i+1}n_i$ and so $m_i = n_{i-1} -xn_i +h_i$ with $h_i \in 0:_M x^i$ for all $i \geq 0$ with $n_{-1} = 0$ and $h_0 = 0$. 
	Now let $n(U) = \sum_{i \geq 0} n_iU^i, h(U) = \sum_{i \geq 1}  h_iU^i$ and therefore $m(U) = (x-U) n(U) + h(U)$. It will be enough to show that the class of $h(U)$ belongs to $\Ker f$. Since $xh_1 = 0$ we have that 
	$(x h(U))/U \in M[[U]]$ and therefore
	$$
	h(U) - (x-U) (h(U)/U) = {\textstyle \sum}_{i \geq 1} xh_{i+1}U^i \;\text{ with }\; x h_{i+1} \in 0:_M x^i.
	$$ 
	Whence it will be enough to show that $\sum_{i \geq 1} xh_{i+1}U^i$ maps to zero where $xh_{i+1} \in 0:_M x^i$. 
	Iterating this process it will be enough to show that the class  $\sum_{i \geq 1} x^kh_{k+i} U^i$ belongs to
	$ \ker f$, where $x^kh_{k+i} \in 0:_M x^i$. Since $M$ is of bounded 
	$x$-torsion $0:_M x^k = 
	0:_M x^{k+i}$ for some $k \in \mathbb{N}$ and all $i \geq 1$. That is, 
	the last formal power series is zero. 
	This proves $H_0(K_{\bullet}(x-U;M[[U]])) \cong \hat{M}^{xR}$. 
\end{proof} 

As a another result we summarize some results about the Koszul homology

\begin{lemma} \label{hoc-3}
	For an $R$-module $X$ and a system of elements $\xx = x_1,\ldots,x_r$ we have the isomorphisms
	\begin{itemize}
		\item[(a)] $H_i(\xx-\UU;X[[\UU]]) \cong \Tor ^{R[\UU]}_ i(R[\UU]/(\xx-\UU)R[\UU], X[[\UU]])$,
		\item[(b)] $H_i(\xx-\UU;X[[\UU]]) \cong \Ext_{R[\UU]}^{r-i}(R[\UU]/(\xx-\UU)R[\UU], X[[\UU]])$
	\end{itemize}
	for all $i \in \mathbb{Z}$.
\end{lemma}

\begin{proof}
	First note that $K_{\bullet}(\xx-\UU;X[[\UU]]) \cong K_{\bullet}(\xx-\UU;R[\UU]) \otimes_{R[\UU]} X[[\UU]]$. Because $\xx-\UU$ 
	forms an $R[\UU]$-regular sequence the Koszul complex $K_{\bullet}(\xx-\UU;R[\UU])$ is a free $R[\UU]$-resolution of $R[\UU]/(\xx-\UU)R[\UU]$. Whence the claim in (a) follows. By the self-duality of the Koszul complex this implies also (b). 
\end{proof}

\begin{remark} \label{hoc-4}
	(A) Let $\UU = U_1,\ldots,U_r$ denote a family of variables over the commutative ring $R$. For an $R$-module $M$ it is clear that $M[[\UU]] = \varprojlim M[\UU]/\UU^{(n)}M[\UU]$. Note that the topology defined by the ideals $\UU^{(n)}R[\UU]$ 
	is equivalent to the $\UU R[\UU]$-adic topology. 
	This construction extents to $X[[\UU]] \cong \varprojlim X[\UU]/\UU^{(n)}X[\UU]$ for 
	an $R$-complex $X$. Note that $X[[\UU]]$ has also the structure of an $R[\UU]$-complex.
	\\
	(B) For a system of elements $\xx = x_1,\ldots,x_r$ of a commutative ring and an $R$-complex 
	$X$ we have that $\Hom_R(\mathcal{L}_{\xx},X) \cong K_{\bullet}(\xx-\UU; X[[\UU]])$. 
	With the previous investigations there is the following 
	isomorphism 
	\[
	K_{\bullet}(\xx-\UU; X[[\UU]]) \cong \varprojlim K_{\bullet}(\xx-\UU; X[\UU]/\UU^{(n)}X[\UU]).
	\] 
	(C) Because the  morphisms in the inverse system $\{K_{\bullet}(\xx-\UU; X[\UU]/\UU^{(n)}X[\UU])\}$ are degree-wise surjective there are the following short 
	exact sequences
	\begin{gather*}
		0 \to \varinjlim{}^1 H_{i+1}(\xx-\UU;X[\UU]/\UU^{(n)}X[\UU]) 
		\to H_i(\xx-\UU; X[[\UU]])  \\\to 
		\varinjlim H_i(\xx-\UU;X[\UU]/\UU^{(n)}X[\UU]) \to 0
	\end{gather*}
	for an $R$-complex $X$ (see \cite[Lemma 1.2.8]{SS}). 
	So there is some interest in the study of the homology of $K_{\bullet}(\xx-\UU; X[\UU]/\UU^{(n)}X[\UU])$ as done in the following sections.
\end{remark}

\section{Weakly pro-regular sequences}
\begin{z} \label{weak-1} {\it Left derived functors of completion.}
	Let $X$ denote an $R$-complex. Let $\xx = x_1,\ldots,x_r$ denote a weakly 
	pro-regular sequence and $\mathfrak{a} = \xx R$. Then it is known that 
	$\Hom_R(\check{L}_{\xx},X)$ is a representative of ${\rm L} \Lambda^{\mathfrak{a}}(X)$ in the 
	derived category (see \cite{GM} and \cite{ALL1}, \cite{Sp2}, \cite{SS} 
	for the generalization to complexes). We refer also to \cite{SS} for more detailed description 
	as well as some further comments. 
	
	Furthermore, let $\xx$ be weakly 
	pro-regular then $H_i(\Hom_R(\check{L}_{\xx},X)) \cong 
	\Lambda_i^{\mathfrak{a}}(X)$ for all $i \in \mathbb{Z}$. Here 
	$ \Lambda_i^{\mathfrak{a}}(X)$ denotes the $i$-th left derived functor 
	of the $\mathfrak{a}$-adic completion. It is 
	defined in the following way: Let $F \quism X$ where $F$ is a $K$-flat 
	resolution of $X$ in the sense of Spaltenstein (see \cite{Sn}). Then 
	$\Lambda^{\mathfrak{a}}(F) = \varprojlim (F \otimes_R R/\mathfrak{a}^n)$ 
	is a representative of ${\rm L} \Lambda^{\mathfrak{a}}(X)$ in the derived category. 
\end{z}

\begin{theorem} \label{weak-2}
	Let $\xx = x_1,\ldots,x_r$ denote a weakly pro-regular sequence 
	and $\mathfrak{a} = \xx R$. Then there are isomorphisms 
	\[
	\Lambda_i^{\mathfrak{a}}(X) \cong H_i(\xx-\UU;X[[\UU]]) 
	\]
	for all $i \in \mathbb{Z}$. More general, $K_{\bullet}(\xx-\UU;X[[\UU]])$ is a
	representative of ${\rm L} \Lambda^{\mathfrak{a}}(X)$ in the derived 
	category.
\end{theorem}

\begin{proof}
	If $\xx$ is a weakly pro-regular sequence, then  
	$\Hom_R(\check{L}_{\xx},X)$ is a representative of 
	${\rm L} \Lambda^{\mathfrak{a}}(X)$ in the derived category. But 
	there is a quasi-isomorphism $\check{L}_{\xx} \quism \mathcal{L}_{\xx}$ 
	of bounded complexes of free $R$-modules. It induces 
	a quasi-isomorphism $\Hom_R(\check{L}_{\xx},X) \quism \Hom_R(\mathcal{L}_{\xx},X)$. By view of \ref{hoc-1} (B) this proves the statements.
\end{proof}

The previous result illustrates the importance of weakly pro-regular sequences. We 
shall continue with their study and a generalization of \cite[Lemma 2.4]{Sp2} to the case of 
$R$-modules.

\begin{proposition} \label{weak-3}
	Let $M$ be an $R$-module. Let $\xx = x_1,\ldots,x_r$ denote a system 
	of elements of $R$. Then the following conditions are equivalent:
	\begin{itemize}
		\item[(i)] $\xx$ is $M$-weakly pro-regular.
		\item[(ii)] $\{H_i(x^{(n)};M \otimes_R F)\}$ is pro-zero for all $i > 0$ 
		and any flat $R$-module $F$.
		\item[(iii)] $\varinjlim H^i(\xx^{(n)};\Hom_R(M,I)) = 0$ for all $i > 0$ 
		and any injective $R$-module $I$.
		\item[(iv)] $H^i(\check{C}_{\xx} \otimes_R \Hom_R(M,I)) = 0$ for all 
		$i > 0$ and any injective $R$-module $I$. 
	\end{itemize}
\end{proposition}

\begin{proof} 
	The equivalence (i) $\Leftrightarrow$ (ii) is clear since   
	$H_i(\xx^{(n)};M) \otimes_R F \cong H_i(\xx^{(n)} ; M\otimes_RF)$
	for all $i$ when $F$ is a flat $R$-module. 
	The equivalence (iii) $\Leftrightarrow$ (iv) follows since 
	$\varinjlim H^i(\xx^{(n)};\Hom_R(M,I)) =H^i(\Cech \otimes_R \Hom_R(M,I))$ 
	because $\varinjlim K^{\bullet}(\xx^{(n)};\Hom_R(M,I)) \cong \check{C}_{\xx} 
	\otimes_R\Hom_R(M,I)$.  
	
	Now let us prove (i) $\Rightarrow$ (iii). Since $I$ is an injective $R$-module 
	\[ 
	H^i(\Hom_R(K_{\bullet}(\xx^{(n)};M), I)) \cong \Hom_R(H_i(\xx^{(n)};M), I)
	\] 
	for all $i$. Therefore, and because of $H^i(\Hom_R(K_{\bullet}(\xx^{(n)};M), I)) \cong H^i(\xx^{(n)}; \Hom_R(M,I))$, 
	\[
	\varinjlim H^i(\xx^{(n)}; \Hom_R(M,I)) \cong \varinjlim \Hom_R (H_i(\xx^{(n)};M), I).
	\] 
	By the assumption the inverse system $\{H_i(\xx^{(n)};M)\}$ is pro-zero for $i \not= 0.$ 
	Whence the direct limit $\varinjlim H^i(\xx^{(n)}; \Hom_R(M,I))$ vanishes, as required. 
	
	In order to complete the proof we have to show that (iii) $\Rightarrow$ (i). 
	Let $f : H_i(\xx^{(n)};M) \hookrightarrow I$ denote an injection into an 
	injective $R$-module $I.$ Then 
	\[ 
	f \in \Hom_R(H_i(\xx^{(n)};M), I) \cong H^i(\xx^{(n)}; \Hom_R(M,I))
	\]
	since $I$ is an injective $R$-module. Because of the assumption we have the 
	vanishing 
	$$
	\varinjlim H^i(\xx^{(n)}; \Hom_R(M,I)) = 0.
	$$ 
	So there must be an integer $m \geq n$ such 
	that the image of $f$ in $H^i(\xx^{(m)}; \Hom_R(M,I))$ has to be zero. In other words, the composite of the maps 
	\[
	H_i(\xx^{(m)};M) \to H_i(\xx^{(n)};M) \stackrel{f}\hookrightarrow I 
	\]
	is zero. Since $f$ is an injection it follows that the first map has to 
	be zero and $\{H_i(\xx^{(n)};M)\}$ is pro-zero for all $i > 0$.
\end{proof}

Note that the case of $M = R$, that is, then $\xx$ is weakly pro-regular. It has 
the consequence that $H^i(\check{C}_{\xx} \otimes_R I) = 0$ for all 
$i > 0$. This is part of the following corollary.

\begin{corollary} \label{weak-4}
	Let $\xx$ denote an $M$-weakly pro-regular sequence and $\mathfrak{a} = 
	\xx R$. Then the natural map 
	\[
	\Gamma_{\mathfrak{a}}(\Hom_R(M,I)) \to \check{C}_{\xx} \otimes_R \Hom_R(M,I)
	\]
	is a quasi-isomorphism for any injective $R$-module $I$. 
\end{corollary}

\begin{proof}
	Clearly $H^0(\check{C}_{\xx} \otimes_R \Hom_R(M,I)) \cong \Gamma_{\mathfrak{a}}(\Hom_R(M,I))$ and $\check{C}_{\xx} \otimes_R \Hom_R(M,I)$ is a left resolution of $H^0(\check{C}_{\xx} \otimes_R \Hom_R(M,I))$ as follows by view of \ref{weak-3}. 
\end{proof}

Note that the in case of $M = R$ and an injective $R$-module $I$ it follows that 
$\Gamma_{\mathfrak{a}}(I) \cong \check{C}_{\xx}(I)$.  This implies that $H^i_{\mathfrak{a}}(X) \cong H^i(\check{C}_{\xx} \otimes_RX)$ for $\mathfrak{a} = \xx R$ and an $R$-complex $X$ (see e.g. \cite[Section 7.4]{SS}). This does 
not hold in general (see e.g. \cite[Example A.2.4]{SS}).

Let $\xx = x_1,\ldots,x_r$ a system of elements of $R$ and let $\UU = U_1,\ldots, U_r$ 
be a set of variables over $R$. Next we will come back to the study of the Koszul 
complexes $K_{\bullet}(\xx-\UU; X[\UU]/\UU^{(n)}X[\UU])$ for an $R$-complex $X$ and an integer $n \in \mathbb{N}_+$. 

\begin{theorem} \label{weak-5}
	Let $\xx = x_1,\ldots,x_r$ denote a system of elements of the commutative ring $R$. Let 
	$\UU = U_1,\ldots,U_r$ be a set of variables over $R$. For an $R$-complex $X$ 
	and an integer $n \geq 1$ there is a quasi-isomorphism of $R$-complexes 
	\[
	K_{\bullet}(\xx-\UU;X[\UU]/\UU^{(n)}X[\UU]) \simeq K_{\bullet}(\xx^{(n)};X)
	\]
	such that for positive integers $m \geq n$ there is a commutative diagram 
	\[
	\xymatrix{
	H_i(\xx-\UU;X[\UU]/\UU^{(m)}X[\UU]) \ar[d] \ar[r]^-{\cong} & H_i(\xx^{(m)};X) \ar[d]\\
	H_i(\xx-\UU;X[\UU]/\UU^{(n)}X[\UU])  \ar[r]^-{\cong}& H_i(\xx^{(n)};X), 
	}
	\]
	for all $i \in \mathbb{Z}$, where the vertical homomorphisms are induced by the natural ones. 
\end{theorem}

\begin{proof}
	Let $f_n(T,U) = (T^n-U^n)/(T-U) = T^{n-1} +T^{n-2} U + \ldots + U^{n-1} \in R[T,U]$ for $n \geq 1$. Then 
	\[
	x_i^n = (x_i-U_i)f_n(x_i,U_i) +U_i^n \mbox{ for all } n \geq 1 \mbox{ and } i = 1,\ldots,r.
	\]
	Therefore there is an invertible $2r \times 2r$-matrix $\mathfrak{A}$ such that 
	$(\xx-\UU,\xx^{(n)}) = (\xx-\UU,\UU^{(n)}) \cdot \mathfrak{A}$. Let $m \geq n$ be two positive integers. By \ref{prel-3} there is a commutative diagram 
	\[
	\xymatrix{
	K_{\bullet}(\xx-\UU,\UU^{(m)};X[\UU])  \ar[d] \ar[r]^-{\cong}& 
	K_{\bullet}(\xx-\UU,\xx^{(m)};X[\UU]) \ar[d]\\
	K_{\bullet}(\xx-\UU,\UU^{(n)};X[\UU])  \ar[r]^-{\cong} & 
	K_{\bullet}(\xx-\UU,\xx^{(n)};X[\UU]), 
	}
	\]
	where the vertical morphisms are the natural ones. Now there are isomorphisms 
	\[
	K_{\bullet}(\xx-\UU,\UU^{(n)};X[\UU]) \cong K_{\bullet}(\xx-\UU;R[\UU])
	\otimes_{R[\UU]} K_{\bullet}(\UU^{(n)};X[\UU]) 
	\]
	for all $n \geq 1$.
	By an iterative use of \ref{prel-7} (a) there is a quasi-isomorphism 
	$K_{\bullet}(\UU^{(n)};X[\UU]) \to X[\UU]/\UU^{(n)}X[\UU]$. Since $K_{\bullet}(\xx-\UU;R[\UU])$ is a complex of free $R[\UU]$-modules it induces a commutative diagram 
	\[
	\xymatrix{
		K_{\bullet}(\xx-\UU,\UU^{(m)};X[\UU])  \ar[d] \ar[r]& K_{\bullet}(\xx-\UU;
		X[\UU]/\UU^{(m)}X[\UU]) \ar[d] \\
		K_{\bullet}(\xx-\UU,\UU^{(n)};X[\UU])  \ar[r]& K_{\bullet}(\xx-\UU;
		X[\UU]/\UU^{(m)}X[\UU]) 
	}
	\]
	where the horizontal morphisms are quasi-isomorphisms and the vertical maps are the natural ones. 
	
	Next recall that $X[\UU]/(\xx-\UU)X[\UU] \cong X$ for an $R$-complex $X$ as follows by 
	degree-wise inspection. By a similar argument as before we get a commutative 
	diagram 
	\[
	\xymatrix{
		K_{\bullet}(\xx-\UU,\xx^{(m)};X[\UU]) \ar[r] \ar[d] & K_{\bullet}(\xx^{(m)};X) \ar[d]\\
		K_{\bullet}(\xx-\UU,\xx^{(n)};X[\UU]) \ar[r]  & K_{\bullet}(\xx^{(n)};X)
	}
	\]
	where the horizontal morphisms are quasi-isomorphisms and the vertical maps are 
	the natural ones. Putting together the isomorphism with the quasi-isomorphisms 
	this proves the claims.
\end{proof}

As an application we get the following corollary for the computation of the 
homology of $K_{\bullet}(\xx-\UU;X[[\UU]])$. Note that the quasi-isomorphisms in \ref{weak-5} does not imply a quasi-isomorphism for the inverse limit.

\begin{corollary} \label{weak-6}
	With the previous notation there is a quasi-isomorphism 
	\[
	K_{\bullet}(\xx-\UU;X[[\UU]]) \simeq M(\xx;X).
	\]
	and there are short exact sequences
	\[
	0 \to \varprojlim{}^1 H_{i+1}(\xx^{(n)};X) \to H_i(\xx-\UU;X[[\UU]]) \to 
	\varprojlim H_i(\xx^{(n)};X) \to 0
	\] 
	for all $i \in \mathbb{Z}$.
\end{corollary}

\begin{proof} 
	The quasi-isomorphisms shown in \ref{weak-5} induce a quasi-isomorphism 
	for the corresponding microscopes  
	\[
	\Mic(\{K_{\bullet}(\xx-\UU;X[\UU]/\UU^{(n)}X[\UU]) \}) \simeq 
	\Mic(\{K_{\bullet}(\xx^{(n)};X) \}) = M(\xx;X)
	\]
	(see \ref{hoc-4} (D)). 
	Now we recall that $K_{\bullet}(\xx-\UU;X[[\UU]]) \cong \varprojlim 
	K_{\bullet}(\xx-\UU;X[\UU]/\UU^{(n)}X[\UU])$. Moreover, the inverse system in the previous system is degree-wise surjective.  By \ref{hoc-4} (B) we get 
	a quasi-isomorphism $K_{\bullet}(\xx-\UU;X[[\UU]]) \simeq 
		\Mic(\{K_{\bullet}(\xx-\UU;X[\UU]/\UU^{(n)}X[\UU]) \})$, which proves the first statement. 
	By \ref{hoc-4} (C) there are short exact sequences 
	\[
	0 \to \varinjlim{}^1 H_i(\xx-\UU;X[\UU]/\UU^{(n)}X[\UU]) \to H_i(\xx-\UU;X[[\UU]]) \to \varprojlim H_i(\xx-\UU;X[\UU]/\UU^{(n)}X[\UU]) 
	\to 0
	\]
	for all $i \in \mathbb{Z}$. By \ref{weak-5} there is an isomorphism of inverse systems 
	\[
	\{H_i(\xx^{(n)};X)\}_{n \geq 1} \cong \{H_i(\xx-\UU;X[\UU]/\UU^{(n)}X[\UU])\}_{n \geq 1}
	\]
	and therefore $\varprojlim H_i(\xx^{(n)};X) \cong \varprojlim H_i(\xx-\UU;X[\UU]/\UU^{(n)}X[\UU])$ for all $i \in \mathbb{Z}$. 
	This completes the proof.
\end{proof}

In the following there are further properties of $M$-weakly pro-regular 
sequences.

\begin{corollary} \label{weak-7}
	Let $\xx = x_1,\ldots,x_r$ denote a system of elements of a commutative 
	ring $R$, and let $M$ be an $R$-module. Let $\hat{M}^{\mathfrak{a}}$ denote the $\mathfrak{a}$-adic completion with $\mathfrak{a} = \xx R$.  
	\begin{itemize}
		\item[(a)] $\xx$ is an $M$-weakly pro-regular sequence if and only if 
		the inverse system $$\{H_i(\xx-\UU;M[\UU]/\UU^{(n)}M[\UU])\}_{n \geq 1}$$ is pro-zero for all $i >0$.
		\item[(b)] There is an isomorphism and an epimorphism  
		$$
		H_0(\xx-\UU;M[[\UU]]) \cong M[[\UU]]/(\xx-\UU)M[[\UU]] \twoheadrightarrow \hat{M}^{\mathfrak{a}}.
		$$
		The epimorphism is an isomorphism if and only if 
		$\varprojlim{}^1 H_1(\xx^{(n)};M) =0$.
		\item[(c)] If $\xx = x_1,\ldots,x_r$ is an $M$-weakly 
		pro-regular sequence, then 
		\[
		H_0(\xx-\UU;M[[\UU]]) \cong \hat{M}^{\mathfrak{a}} \mbox{ and } 
		H_i(\xx-\UU;M[[\UU]]) = 0 \mbox{ for all } i >0.
		\]
		That is, $K_{\bullet}(\xx-\UU;M[[\UU]])$ is a left resolution of $\hat{M}^{\mathfrak{a}}$.
	\end{itemize}
\end{corollary}

\begin{proof}
	(a): There are isomorphisms of inverse systems 
	$$
	\{H_i(\xx^{(n)};M)\}_{n \geq 1} \cong \{H_i(\xx-\UU;X[\UU]/\UU^{(n)}M[\UU])\}_{n \geq 1} \mbox{ for all } i \geq 0
	$$ 
	with the corresponding inverse maps
	(see \ref{weak-5}). Then the claim follows by the definition of the notion of an $M$-weakly pro-regular sequence. \\
	(b): By virtue of \ref{weak-6} there is a short exact sequence 
	\[
	0 \to \varprojlim{}^1 H_1(\xx^{(n)};M) \to H_0(\xx-\UU;M[[\UU]]) 
	\to \varprojlim H_0(\xx^{(n)};M) \to 0.
	\]
	Because  $H_0(\xx-\UU;M[[\UU]]) \cong M[[\UU]]/(\xx-\UU)M[[\UU]]$ 
	as it is clear by the Koszul homology and because  
	$$
	\varprojlim H_0(\xx^{(n)};M) \cong \varprojlim M/\xx^{(n)}M \cong \hat{M}^{\mathfrak{a}}
	$$
	the statements follow. \\
	(c): If $\xx$ is an $M$-weakly pro-regular sequence, then $\{H_i(\xx^{(n)};M)\}_{n \geq 1}$ is pro-zero for all $i > 0$. Therefore,
	$\varprojlim H_i(\xx^{(n)};M) = \varprojlim{}^1 H_i(\xx^{(n)};M) = 0$ for 
	all $i \geq 1$. Then the statement follows by \ref{weak-6} and (b).
\end{proof}

For the next result let $\xx = x_1,\ldots,x_r$ be sequence in $R$ and $\mathfrak{a} = \xx R$. The statement in \ref{weak-7} (b) yields that $\hat{R}^{\mathfrak{a}}$ 
has a natural structure as an $R[[\UU]]$-module given by 
\[
R[[\UU]] \to \hat{R}^{\mathfrak{a}}, \; r(U_1,\ldots,U_r) \mapsto r(x_1,\ldots,x_r).
\]
By \ref{weak-7} (c) $K_{\bullet}(\xx-\UU;R[[\UU]])$ is a bounded left $R[[\UU]]$-free resolution of $\hat{R}^{\mathfrak{a}}$ 
and $H_0(\xx-\UU;R[[\UU]]) \cong \hat{R}^{\mathfrak{a}}$ provided 
$\xx$ is a weakly pro-regular sequence. 

\begin{corollary} \label{weak-9}
	Let $\xx = x_1,\ldots,x_r$ be a weakly pro-regular sequence and $\mathfrak{a} = \xx R$. Let $X$ denote an $R$-complex. Then there are isomorphisms 
	\[
	\Lambda_i^{\mathfrak{a}}(X) \cong H_i(\xx-\UU;X[[\UU]]) \cong 
	\Tor_i^{R[[\UU]]}(\hat{R}^{\mathfrak{a}},X[[\UU]]) 
	\cong \Ext_{R[[\UU]]}^{r-i}(\hat{R}^{\mathfrak{a}},X[[\UU]]) 
	\]
	for all $i \in \mathbb{Z}$. In particular, for an $R$-module $M$ it follows 
	that $\Lambda_0^{\mathfrak{a}}(M) \cong \hat{R}^{\mathfrak{a}} \otimes_{R[[\UU]]} M[[\UU]]$. 
\end{corollary}

\begin{proof}
	The first isomorphism follows by \ref{weak-2}.
	There is the following isomorphism of Koszul complexes 
	\[
	 K_{\bullet}(\xx-\UU;X[[\UU]]) \cong K_{\bullet}(\xx-\UU;R[[\UU]]) 
	 \otimes_{R[[\UU]]} X[[\UU]].
	\]
	Since $K_{\bullet}(\xx-\UU;R[[\UU]])$ is a left $R[[\UU]]$-free resolution 
	of $\hat{R}^{\mathfrak{a}}$ as $R[[\UU]]$-module. Then the second isomorphisms 
	follow by the definition. The last isomorphisms are a consequence of the self-duality of the Koszul complex.
\end{proof}

In the following we discuss the necessity of the condition in \ref{weak-7} (b).
The following example was suggested by Anne Marie Simon.  

\begin{example} \label{weak-8}
	(A) Let $\Bbbk$ denote a field and let $R = \Bbbk[[x]]$ denote the formal 
	power series ring in the variable $x$ over $\Bbbk$. Let $M = \prod_{i \geq 1} 
	R/x^iR$. Then $M$ is -- as a direct product of $xR$-complete modules -- 
	$xR$-adic complete (see \cite[2.2.7]{SS}). 
	
	By \ref{weak-7} and \ref{weak-8} we see that $H_0(x-U;M[[U]]) \cong M$ and 
	$H_i(x-U;M[[U]]) = 0$ for $i >0$. But $M$ is not of bounded $xR$-torsion 
	as follows since $\oplus_{i \geq 1} R/x^iR \subset M$. \\
	(B) Another feature of the Example in (A) is the following: Note that 
	an inverse system of $R$-modules $\{M_n\}_{n \in \mathbb{N}}$ is called 
	pro-zero if for any integer $n$ there is an integer $m \geq n$ such that 
	the transition map $M_m \to M_n$ is zero. If $\{M_n\}_{n \in \mathbb{N}}$  
	is pro-zero, then 
	\[
	\varprojlim M_n = \varprojlim{}^1 M_n = 0
	\]
	(see e.g. \cite[1.2.4]{SS}). \\
	(C) It is natural to ask whether the converse of the implication in (B) 
	is true. This is not the case. To this end we investigate the example in (A). By the conclusions in (A) and by view of  \ref{weak-6} it follows that 
	\[
	\varprojlim H_1(x^n;M) = \varprojlim{}^1 H_1(x^n;M) = 0,
	\]
	where $\{H_1(x^n;M)\}_{n \geq 1}$ is the inverse system with 
	$H_1(x^n;M) = 0:_M x^n$ and the transition map 
	$$
	\mu_{m-n} : 0:_M x^m \to 0:_M x^n
	$$ 
	is the multiplication by $x^{m-n}$ for $m \geq n$. Because $M$ is not of bounded $xR$-torsion it follows that the inverse system $\{0:_M x^n\}_{n \geq 1}$ is not pro-zero.
\end{example}

\section{\v{C}ech cohomology via Koszul complexes}
\begin{z} \label{coh-1} {\it Inverse polynomials.}
	(A) Let $R$ denote a commutative ring. Let $R[\UU]$ denote the 
	polynomial ring in the variables $\UU = U_1,\ldots,U_r$ over $R$. 
	We are interested in the \v{C}ech complex $\check{C}_{\UU}(R[\UU])$. 
	We have the isomorphism of $R[\UU]$-complexes
	\[
	\check{C}_{\UU}(R[\UU]) \cong \varinjlim K^{\bullet}(\UU^{(n)};R[\UU]).
	\]
	Since $\UU$ forms an $R[\UU]$-regular sequence $K^{\bullet}(\UU^{(n)};R[\UU])$, is a left resolution of $H^r(\UU^{(n)};R[\UU] \cong R[\UU]/\UU^{(n)}R[\UU]$ by free $R$-modules.
	Moreover, the \v{C}ech complex $\check{C}_{\UU}(R[\UU])$ is 
	a flat resolution of $H^r(\check{C}_{\UU}(R[\UU]))$. That is, there is a quasi-isomorphism $\check{C}_{\UU}(R[\UU]) \simeq H^r(\check{C}_{\UU}(R[\UU]))^{[-r]}$.
	Because of the natural map $K^{\bullet}(\UU^{(n)};R[\UU]) \to 
	K^{\bullet}(\UU^{(n+1)};R[\UU])$  
	it follows that 
	\[
	\check{C}_{\UU}(R[\UU]) \cong \varinjlim K^{\bullet}(\UU^{(n)};R[\UU]), 
	\]
	and therefore $H^r(\check{C}_{\UU}(R[\UU])) \cong \varinjlim \{H^r(\UU^{(n)};R[\UU]), U_1\cdots U_r\}$.
	Note that direct limits commute with homology and the natural map 
	$H^r(\UU^{(n)};R[\UU]) \to H^r(\UU^{(n+1)};R[\UU])$ is the multiplication 
	by $U_1\cdots U_r$.\\
	(B) An interpretation of $H^r(\check{C}_{\UU}(R[\UU]))$ 
%is given 
%	by the notion 
%	of inverse polynomials as introduced by Macaulay (see \cite{Mfs}). 
	is shown by Brodmann and Sharp (see \cite[13.5]{BrS}).
	The \v{C}ech complex $\check{C}_{\UU}(R[\UU])$ is a resolution of 
	$H^r(\check{C}_{\UU}(R[\UU]))^{[-r]}$ because $\UU$ is an $R[\UU]$-regular sequence. By an inspection of the complex it follows that  
	\[
	H^r(\check{C}_{\UU}(R[\UU])) \cong R[\UU^{-}].
	\]
	Here $R[\UU^{-}]$ denotes the free $R$-module with basis $\{U_1^{-i_1}\cdots 
	U_r^{-i_r}, (i_1,\ldots,i_r) \in \mathbb{N}_+^r\}$, where $\mathbb{N}_+^r = 
	(\mathbb{N}_+)^r$. Its $R[\UU]$-module structure is given by 
	\[
	U_k\cdot (U_1^{-i_1}\cdots U_r^{-i_r}) = 
		\begin{cases} 
		U_1^{-i_1}\cdots U_k^{-i_k+1} \cdots U_r^{-i_r} & \mbox{ if } i_k < 1\\
		0 & \mbox{ if } i_k = 1.
		\end{cases}
	\]
	Moreover, let $R[\UU^{-1}] = R[U_1^{-1},\ldots,U_r^{-1}]$ denote the polynomial ring over $R$ in $U_1^{-1},\ldots,U_r^{-1}$. We recall that 
	$R[\UU^{-1}]$ is an $R[\UU]$-module where the product  of 
	$rU_1^{j_1}\cdots U_r^{j_r} $ and $sU_1^{-i_1}\cdots U_r^{-i_r}, r,s \in R,$ is zero 
	unless $j_k \leq i_k$ for $k = 1,\ldots,r,$ in which case it is defined by 
	\[
	(rU_1^{j_1}\cdots U_r^{j_r} ) \cdot (sU_1^{-i_1}\cdots U_r^{-i_r}) = 
	rs U_1^{j_1-i_1}\ldots U_r^{j_r-i_r}.
	\]
	The $R[\UU]$-module $R[\UU^{-1}]$ is called the module of 
	inverse polynomials as 
	introduced by F. S. Macaulay (see \cite{Mfs}). There is an isomorphism 
	of $R[\UU]$-modules 
	\[
		R[\UU^{-}] \cong U_1 \cdots U_r\; R[\UU^{-1}] .
	\]
	Therefore there is an $R[\UU]$-isomorphism $H^r(\check{C}_{\UU}(R[\UU])) 
	\cong R[\UU^{-1}]$.
	
	In the particular case of $R = \Bbbk$ a field. It follows that $\Bbbk 
	[\UU^{-1}]$ as an $\Bbbk [\UU]$-module is the injective hull of $\Bbbk \cong 
	\Bbbk [\UU]/(\UU)$ (see e.g. \cite{Ndg} for some details).
%	For more details we refer to the explanations in \cite[13.5]{BrS}. 
%	Here we remark that 
%	$$
%	R[\UU^{-}] = U_1^{-1}\cdots U_r^{-1} R[U_1^{-1},\ldots,U_r^{-1}] 
%	\cong R[U_1^{-1},\ldots,U_r^{-1}] 
%	$$ 
%	as free $R$-modules. 
	\\
	(C)   For an $R$-module $M$ we define $M[\UU^{-1}] = R[\UU^{-1}] 
	\otimes_R M$. 
	Clearly $M[\UU^{-1}]$ is an $R[\UU]$-module by the structure of $R[\UU^{-1}]$ 
	as $R[\UU]$-module as given before. Let $X$ denote an $R$-complex. Then we 
	define $X[\UU^{-1}]$ degree-wise. Moreover, $X[\UU^{-1}]$ is also an $R[\UU]$-complex. If $X \to Y$ is a quasi-isomorphism, then the induced 
	morphism $X[\UU^{-1}] \to Y[\UU^{-1}]$ is a quasi-isomorphism too.
	Moreover, for an $R$-module $M$ there is an isomorphism 
	\[
	M[\UU^{-1}] \cong R[\UU^{-1}] \otimes_{R[\UU]} M[\UU].
	\]
	This follows since
	\[
	R[\UU^{-1}] \otimes_{R[\UU]} M[\UU] \cong R[\UU^{-1}] \otimes_{R[\UU]} (R[\UU] \otimes_R M) \cong R[\UU^{-1}] \otimes_R M.
	\] 
	It extends to an isomorphism $X[\UU^{-1}] \cong R[\UU^{-1}] \otimes_{R[\UU]} X[\UU]$ for an $R$-complex $X$. Therefore the $R$-complex $X[\UU^{-1}]$ 
	becomes the structure of an $R[\UU]$-complex.
\end{z}

Next we shall give a dual statement to \ref{weak-5} for the Koszul 
cohomology.

\begin{lemma} \label{coh-2}
	Let $\xx = x_1,\ldots,x_r$ denote a system of the commutative ring $R$. Let 
	$\UU = U_1,\ldots,U_r$ be a set of variables over $R$. For an $R$-complex $X$ 
	and an integer $n \geq 1$ there is a quasi-isomorphism of $R$-complexes 
	\[
	K^{\bullet}(\xx-\UU;X[\UU]/\UU^{(n)}X[\UU]) \simeq K^{\bullet}(\xx^{(n)};X)
	\]
	such that for positive integers $m \geq n$ there is a commutative diagram 
	\[
	\xymatrix{
		H^i(\xx-\UU;X[\UU]/\UU^{(n)}X[\UU]) \ar[d]^{\mu_{m-n}} \ar[r]^-{\cong} & H^i(\xx^{(n)};X) \ar[d]\\
		H^i(\xx-\UU;X[\UU]/\UU^{(m)}X[\UU])  \ar[r]^-{\cong}& H^i(\xx^{(m)};X), 
	}
	\]
	for all $i \in \mathbb{Z}$, where the first vertical homomorphism 
	$\mu_{m-n}$ is the multiplication by $U_1^{m-n}\cdots U_r^{m-n}$ and the second is the natural one. 
\end{lemma}

\begin{proof}
	Let $m \geq n$ be two positive integers.
	By the use of the invertible $2r\times 2r$ matrix $\mathfrak{A}$ as defined 
	in the proof of \ref{weak-5} there is a commutative diagram 
	\[
	\xymatrix{
		K^{\bullet}(\xx-\UU,\UU^{(n)};X[\UU])  \ar[d] \ar[r]^-{\cong}& 
		K^{\bullet}(\xx-\UU,\xx^{(n)};X[\UU]) \ar[d]\\
		K^{\bullet}(\xx-\UU,\UU^{(m)};X[\UU])  \ar[r]^-{\cong} & 
		K^{\bullet}(\xx-\UU,\xx^{(m)};X[\UU]), 
	}
	\]
	where the vertical morphisms are the natural ones. Now there are isomorphisms 
	\[
	K^{\bullet}(\xx-\UU,\UU^{(n)};X[\UU]) \cong K^{\bullet}(\xx-\UU;R[\UU])
	\otimes_{R[\UU]} K^{\bullet}(\UU^{(n)};X[\UU]). 
	\]
	By an iterative use of \ref{prel-7} (b) there is a quasi-isomorphism 
	$K^{\bullet}(\UU^{(n)};X[\UU]) \to (X[\UU]/\UU^{(n)}X[\UU])$ such that there is a commutative diagram
	\[
	\xymatrix{
		K^{\bullet}(\UU^{(n)};X[\UU]) \ar[d] \ar[r]^-{\cong} \ar[d] & 
		(X[\UU]/\UU^{(n)}X[\UU]) \ar[d]^{\mu_{m-n}} \\
		K^{\bullet}(\UU^{(m)};X[\UU]) \ar[r]^-{\cong}  & 
		(X[\UU]/\UU^{(m)}X[\UU]).
	}
	\]
 	Since $K_{\bullet}(\xx-\UU;R[\UU])$ is a complex of free $R[\UU]$-modules it induces a commutative diagram 
	\[
	\xymatrix{
		K^{\bullet}(\xx-\UU,\UU^{(n)};X[\UU])  \ar[d] \ar[r]&   
		K^{\bullet}(\xx-\UU;X[\UU]/\UU^{(n)}X[\UU]) \ar[d]^{\mu_{m-n}} \\
		K^{\bullet}(\xx-\UU,\UU^{(m)};X[\UU])  \ar[r]& K^{\bullet}(\xx-\UU;
		X[\UU]/\UU^{(m)}X[\UU]), 
	}
	\]
	where the horizontal morphisms are quasi-isomorphisms and the first vertical map is the natural one and the second, $\mu_{m-n}$, is multiplication by 
	$U_1^{m-n}\cdots U_r^{m-n}$. 
	
	Next recall that $X[\UU]/(\xx-\UU)X[\UU] \cong X$ for an $R$-complex $X$ as follows by 
	degree-wise inspection. By a similar argument as before we get a commutative 
	diagram 
	\[
	\xymatrix{
		K^{\bullet}(\xx-\UU,\xx^{(m)};X[\UU]) \ar[r] \ar[d] & K^{\bullet}(\xx^{(m)};X) \ar[d]\\
		K^{\bullet}(\xx-\UU,\xx^{(n)};X[\UU]) \ar[r]  & K^{\bullet}(\xx^{(n)};X)
	}
	\]
	where the horizontal morphisms are quasi-isomorphisms and the vertical maps are 
	the natural ones. Putting together the isomorphism with the quasi-isomorphisms 
	this proves the claims.
\end{proof}

Now we are ready to prove the main result for the Koszul cohomology.

\begin{theorem} \label{coh-3}
	Let $\xx = x_1,\ldots,x_r$ be a system of elements of $R$. Then there is a 
	quasi-isomorphism 
	\[
	K^{\bullet}(\xx-\UU;X[\UU^{-1}]) \simeq \check{C}_{\xx}(X)
	\]
	for any $R$-complex $X$.Therefore 
	\[
	H^i(\xx-\UU;X[\UU^{-1}]) \cong H^i(\check{C}_{\xx}(X))
	\]
	for all $i \in \mathbb{Z}$. In the derived category $K^{\bullet}(\xx-\UU;X[\UU^{-1}])$ is a representative of $\check{C}_{\xx}(X)$.
\end{theorem}

\begin{proof}
	By \ref{coh-2} the direct systems $\{K^{\bullet}(\xx-\UU;X[\UU]/\UU^{(n)}X[\UU])\}$ and $\{K^{\bullet}(\xx^{(n)};X)\}$ are quasi-isomorphic. Since the direct limit preserves quasi-isomorphisms there is a quasi-isomorphism 
	\[
	\varinjlim K^{\bullet}(\xx-\UU;X[\UU]/\UU^{(n)}X[\UU]) \simeq \check{C}_{\xx}(X).
	\]
	In order to conclude we have to show that the direct limit at the left 
	is isomorphic to $K^{\bullet}(\xx-\UU;X[\UU^{-1}])$. But this follows by properties of Koszul complexes since 
	$X[\UU^{-1}] \cong \varinjlim X[\UU]/\UU^{(n)}X[\UU]$ with the multiplication map by $U_1^{m-n}\cdots U_r^{m-n}$ (see \ref{coh-1}).
	The additional statements follow immediately. 
\end{proof}

\begin{z} {\it Alternative proof of \ref{coh-3}.} \label{coh-4}
	(A) We consider the Koszul complex 
	\[
	K^{\bullet}(x-U;R[U^{-1}]):   0 \to R[U^{-1}] \stackrel{x-U}{\longrightarrow} R[U^{-1}] \to 0 
	\]
	for a single element $x \in R$ and $R[U^{-1}]$. Here $R[U^{-1}]$ is the system 
	of inverse polynomials in the variable $U$ as $R[U]$-module.
	As above we consider the $R$-complex $\mathcal{L}_x$ (see \ref{comp-4}) with the variable $U^{-1}$ instead of $U$. Then we construct an isomorphism of $R$-complexes
	\[
	\xymatrix{
		K^{\bullet}(x-U;R[U^{-1}]): \ar[d] & 0 \ar[r] & R[U^{-1}] \ar[r]^{x-U} \ar@{=}[d] & 
		R[U^{-1}] \ar[r] \ar[d]^{U^{-1}} & 0\\
		\mathcal{L}_x(R): & 0 \ar[r] & 	
		R[U^{-1}] \ar[r]^{\psi}  & 
		U^{-1} R[U^{-1}]\ar[r]  & 0,
	}
	\]
	where the left vertical map is multiplication by $U^{-1}$.
	Both of the vertical maps are $R$-isomorphisms and the diagram 
	is commutative as easily seen. Whence we get the isomorphism 
	$$\mathcal{L}_x(R) \cong K^{\bullet}(x-U;R[U^{-1}])$$ 
	and therefore 
	$\mathcal{L}_x(X) \cong  K^{\bullet}(x-U;X[U^{-1}])$ for an $R$-complex $X$. To this end recall that 
	\[
	\mathcal{L}_x(X) \cong \mathcal{L}_x(R) \otimes_RX \mbox{ and } 
	K^{\bullet}(x-U;X[U^{-1}]) \cong K^{\bullet}(x-U;R[U^{-1}]) \otimes_R X.
	\]
	because $R[U^{-1}] \otimes_R X \cong X[U^{-1}]$ for an $R$-complex $X$.\\
	(B) Let $\xx = x_1,\ldots,x_r$ a system of elements of $R$ and $y \in R$ 
	an element. Then we prove 
	$$
	\mathcal{L}_{\xx}(X) \cong K^{\bullet}(\xx-\UU;X[\UU^{-1}])
	$$ 
	by induction on the number of elements. 
	For $r = 1$ this is shown in (A). Now suppose the statement holds for $r$. 
	Then there are isomorphisms 
	\[
	\mathcal{L}_{\xx,y}(X) \cong \mathcal{L}_y(\mathcal{L}_{\xx}(X)) 
	\cong \mathcal{L}_y(K^{\bullet}(\xx-\UU;X[\UU])) \cong 
	K^{\bullet}(y-V;K^{\bullet}(\xx-\UU;X[\UU^{-1}])[V^{-1}]),
	\]	
	where $V$ denotes an additional variable. But now $R[\UU^{-1}]\otimes_R R[V^{-1}] \cong R[(\UU,V)^{-1}]$ and 
	$$
	K^{\bullet}(y-V,\xx-\UU;X[(\UU,V)^{-1}]) 
	\cong K^{\bullet}(\xx-\UU,y-V;X[(\UU,V)^{-1}])
	$$ 
	which completes the inductive step. Note that $\xx-\UU,y-V$ is the sequence $x_1-U_1,\ldots,x_r-U_r,y-V$.
	(C) 
\end{z}

\begin{remark} \label{coh-9}
	(A) Because of the fact that the natural morphism $\mathcal{L}_{\xx}(X) \to 
	\check{C}_{\xx}(X)$ is a quasi-isomorphism this provides another proof of 
	Theorem \ref{coh-3}. The original proof of Theorem \ref{coh-3} uses 
	Lemma \ref{coh-2} which contains additional information about various 
	Koszul complexes that might be of independent interest. \\
	(B) By \ref{coh-3} it follows that there is a quasi-isomorphism $K^{\bullet}(x-U;R[U^{-1}]) \to \check{C}_x$ for an element $x \in R$. By iteration this implies -- as mentioned before - a quasi-isomorphism $K^{\bullet}(\xx-\UU;R[\UU^{-1}]) \to \check{C}_{\xx}$ for a sequence $\xx = x_1,\ldots,x_r$ in $R$. Now note that $K^{\bullet}(\xx-\UU;R[\UU^{-1}])$ 
	is a bounded complex of free $R$-modules. Therefore the Koszul complex  $K^{\bullet}(\xx-\UU;R[\UU^{-1}])$ is a bounded $R$-free resolution of 
	the \v{C}ech complex $\check{C}_{\xx}$.
\end{remark}

In addition to the previous arguments we will add a direct argument in the case of a single element, the initial step. This is more or less a copy of \ref{comp-4}. We include here a direct proof slightly different from this of \ref{comp-4}.

\begin{lemma} \label{coh-8}
	Let $x \in R$ denote an element. Then there is a quasi-isomorphism 
	\[
	\xymatrix{
		K^{\bullet}(x-U;R[U^{-1}]): \ar[d] & 0 \ar[r] & R[U^{-1}] \ar[r]^{x-U} \ar[d]^{f_0} & R[U^{-1}] \ar[r] \ar[d]^{f_1} & 0\\
		\check{C}_x(R): & 0 \ar[r] & 	
		R \ar[r]^{\iota_x}  & 
		R_x \ar[r]  & 0,
	}
	\]
	defined by the maps 
	\[
	f_0 : R[U^{-1}] \to R, \; \textstyle{\sum}_{i=0}^{n} r_i U^{-i} \mapsto r_0, 
	\quad \mbox{ and } \quad 
	f_1 : R[U^{-1}] \to R_x, \; f(U^{-1}) \mapsto f(x^{-1})/x.
	\]
\end{lemma}

\begin{proof}
	First note that $f_1 \cdot \mu_{x-U} =  \iota_x \cdot f_0$ as easily seen. 
	Therefore there are homomorphisms 
	\[
	g_0 : \Ker \mu_{x-U} \to \Ker \iota_x  \quad \mbox{ and } \quad
	g_1 : R[U^{-1}]/(x-U) R[U^{-1}] \to R_x/\iota_x(R).
	\]
	The proof that $g_0$ is an isomorphism follows in a similar way  
	as in the proof of the isomorphism $g_0$ in \ref{comp-4}. Whence we omit 
	the details. 
	
	Now let $r/x^n \in R_x \setminus \iota_x(R)$. Then 
	$r U^{-n+1} + (x-U)R[U^{-1}]$ is mapped to $r/x^n$ so that $g_1$ is onto. 
	Assume now that $r(U^{-1}) +(x-U)R[U^{-1}]$ with $r(U^{-1}) = \sum_{i=0}^{n} r_iU^{-i}$ 
	is mapped to zero under $g_1$. Then it follows $\sum\limits_{i=-1}^{n}r_i/x^{i+1} = 0$ 
	for a certain $r_{-1}/1 \in \iota_x(R)$. So we have  that 
	$x^k(x^{n+1}r_{-1}+ x^nr_0 + \ldots + r_n) = 0$ for a certain integer $k \geq 0$. We define elements 
	$s_i = -(\sum_{j=0}^{i} x^{i-j}r_{j-1})$ for $i = 0,\ldots,n+1$ and 
	$s_i = x s_{i-1}$ for $i = n+2,\ldots,n+k+1$. It follows that 
	$$
	xs_i-s_{i+1} = r_i,\; i= 0,\ldots,n,\; xs_i -s_{i+1} = 0, \; 
	i = n+1,\ldots,n+k, \mbox{ and } x s_{n+k+1} = x^k s_{n+1} =0.
	$$
	Finally we put $s(U^{-1}) = \sum_{i=0}^{n+k+1}s_i U^{-1}$, 
	and it yields that $r(U^{-1}) = (x-U)\cdot s(U^{-1})$ by the choice of 
	the $s_i$'s. Therefore $g_1$ is an isomorphism. 
\end{proof}

\begin{remark} \label{coh-7}
	Let $X \to Y$ denote a quasi-isomorphism of two $R$-complexes. Then there is 
	also a quasi-isomorphism
	\[
	\mathcal{L}_{\xx}  \otimes_R X \cong \mathcal{L}_{\xx}(X) \to 
	\check{C}_{\xx} \otimes_R Y \cong \check{C}_{\xx}(Y). 
	\]
	This follows since $\mathcal{L}_{\xx} \to \check{C}_{\xx}$ is a quasi-isomorphism of bounded complexes of flat $R$-modules. The functor $\Hom_R(\check{C}_{\xx},\cdot)$ does not preserves quasi-isomorphisms, 
	while it is correct for $$\Hom_R(\mathcal{L}_{\xx}(R),\cdot) \cong 
	\Hom_R(K^{\bullet}(\xx-\UU;R[\UU^{-1}]), \cdot)$$ since 
	$\mathcal{L}_{\xx}(R) \cong K^{\bullet}(\xx-\UU;R[\UU^{-1}])$ is an isomorphism of bounded complexes of free $R$-modules.
\end{remark}

As an application there is a dual statement to \ref{hoc-3}.

\begin{corollary} \label{coh-5}
	For an $R$-module $X$ and a system of elements $\xx = x_1,\ldots,x_r$ we have the isomorphisms
	\begin{itemize}
		\item[(a)] $H^i(\check{C}_{\xx}(X)) \cong \Tor^{R[\UU]}_{ r-i}(R[\UU]/(\xx-\UU)R[\UU], X[\UU^{-1}])$,
		\item[(b)] $H^i(\check{C}_{\xx}(X)) \cong \Ext_{R[\UU]}^i(R[\UU]/(\xx-\UU)R[\UU], X[\UU^{-1}])$
	\end{itemize}
	for all $i \in \mathbb{Z}$.
\end{corollary}

\begin{proof}
	By \ref{hoc-3} there are isomorphisms $H^i(\xx-\UU;X[\UU^{-1}]) \cong H^i(\check{C}_{\xx}(X))$ for all $i \in \mathbb{Z}$. For the 
	Koszul complex there is an isomorphism 
	\[
	K^{\bullet}(\xx-\UU;X[\UU^{-1}]) \cong K^{\bullet}(\xx-\UU;R[\UU]) \otimes_{R[\UU]} X[\UU^{-1}].
	\]
	Now $\xx-\UU$ is an $R[\UU]$ regular sequence and $K^{\bullet}(\xx-\UU;R[\UU])$ is an $R[\UU]$-free resolution of 
	$H^r(\xx-\UU;R[\UU]) \cong R[\UU]/(\xx-\UU)R[\UU]$. Then the results follow 
	by the definitions and the self-duality of the Koszul complex.
\end{proof}

Under some additional assumptions about the sequence $\xx$ we get an 
expresion of the local cohomology modules as Koszul cohomology. 

\begin{corollary} \label{coh-6}
	Let $\xx = x_1,\ldots,x_r$ denote a weakly pro-regular sequence 
	and $\mathfrak{a} = \xx R$. Then there are isomorphisms 
	\[
	H^i(\xx-\UU;X[\UU^{-1}]) \cong H^i_{\mathfrak{a}}(X)
	\]
	for all $i \in \mathbb{Z}$. More general $K^{\bullet}(\xx-\UU; X[[\UU^{-1}]])$ 
	is a representative of ${\rm R} \Gamma_{\mathfrak{a}}(X)$ in the derived 
	category.
\end{corollary}

\begin{proof}
	If $\xx$ is weakly pro-regular we know that $\check{C}_{\xx}(X)$ is a representative of ${\rm R} \Gamma_{\mathfrak{a}}(X)$ (see e.g.\cite{SS}). Then the claims are a consequence of \ref{coh-5}.
\end{proof}

\section{Duality}
There is a well-known duality between Koszul complexes and 
Koszul co-complexes. In a certain sense there is an extension to the 
polynomial Koszul (co-)complexes as investigated above.

As above let $\xx = x_1,\ldots,x_r$ denote a system of elements of a 
commutative ring $R$. We consider the polynomial ring $R[\UU]$ in $R$ variables 
with the sequence $\xx - \UU = x_1-U_1,\ldots,x_r-U_r$. Moreover, $R[\UU^{-1}]$ 
denotes Macaulay's system of inverse polynomials. At the beginning we need 
a technical result concerning the inverse system.

\begin{proposition} \label{dual-0}
	Fix the previous notation. Then there is an isomorphism of $R[\UU]$-complexes 
	\[
	\Hom_R(R[\UU^{-1}],X) \cong X[[\UU]]
	\]
	for any $R$-complex $X$. 
\end{proposition} 

\begin{proof}
	Let $M$ denote an $R$-module. 
	First of all we define a map $\Phi : \Hom_R(R[\UU^{-1}], M) \to M[[\UU]]$ 
	by means of 
	\[
	\Phi(f) = \sum_{i_1,\ldots,i_r \geq 0} f(U_1^{-i_1}\cdots U_r^{-i_r}) U_1^{i_1}\cdots U_r^{i_r}
	\]
	for $f \in \Hom_R(R[\UU^{-1}], M)$. Evidently $\Phi$ is an isomorphism 
	of $R$-modules. Because $R[\UU^{-1}]$ is an $R[\UU]$ module it follows that $\Hom_R(R[\UU^{-1}], M)$ is an $R[\UU]$-module with 
	\[
	((a U_1^{j_1}\cdots U_r^{j_r}) f)(V) = f(a U_1^{j_1}\cdots U_r^{j_r} V)
	\]
	for $ a \in R, U_1^{j_1}\cdots U_r^{j_r} \in R[\UU],$ and $V \in R[\UU^{-1}]$. Thus 
	\[
	\Phi((a U_1^{j_1}\cdots U_r^{j_r})f) = \sum%_{i_1,\ldots,i_r \geq 0} 
	a f(U_1^{j_1-i_1}\cdots U_r^{j_r-i_r})U_1^{i_1} \cdots U_r^{i–r} = 
	(a U_1^{j_1} \cdots U_r^{j_r}) \Phi(f).
	\]
	Therefore the map $\Phi$ is an homomorphism of $R[\UU]$-modules where 
	the structure of $M[[\UU]]$ as $R[\UU]$-module is the natural one. Let $m(\UU) = \sum m_{i_1\ldots i_r} U_1^{i_1}\cdots U_r^{i_r} \in M[[\UU]]$. It defines $f(m) \in \Hom_R(R[\UU^{-1}],M)$ 
	by $f(m)(U_1^{-i_1}\cdots U_r^{-i_r}) = m_{i_1\ldots i_r}$. Therefore, it becomes an isomorphism of $R[\UU]$-modules. 
	
	In order to finish the claim note that the isomorphism $\Phi$ extents 
	degree-wise to an $R$-complex $X$.
\end{proof}

%\begin{proof}
%	We have the $R$-isomorphism $R[\UU^{-}] 
%	\cong R[U_1^{-1},\ldots,U_r^{-1}]$ (see \ref{coh-1} (B))  of free $R$-modules. It implies the first isomorphism of the following 
%	$R$-complexes 
%	\[
%	\Hom_R(R[\UU^{-}],X) \cong \Hom_R(R[U_1^{-1},\ldots,U_r^{-1}],X) 
%	\cong X[[\UU]].
%	\]
%	The second isomorphism follows because of $R[U_1^{-1},\ldots,U_r^{-1}] \cong R^{(\mathbb{N}^r)}$ and $X[[\UU]] \cong X^{\mathbb{N}^r}$.
%\end{proof}

\begin{theorem} \label{dual-1}
	With the previous notations there is an isomorphism 
	\[
	K_{\bullet}(\xx-\UU;X[[\UU]]) \cong \Hom_R(K^{\bullet}(\xx-\UU;R[\UU^{-1}]),X)
	\]
	for any $R$-module $X$. 
\end{theorem}

\begin{proof}
	By the definition of Koszul complexes there is an isomorphism of $R[\UU]$-complexes 
	\[
	K_{\bullet}(\xx-\UU;X[[\UU]]) \cong \Hom_{R[\UU]}(K^{\bullet}(\xx-\UU;R[\UU]),X[[\UU]]).
	\]
	By view of \ref{dual-0} the second complex is isomorphic to 
	$\Hom_{R[\UU]}(K^{\bullet}(\xx-\UU;R[\UU]), \Hom_R(R[\UU^{-1}],X)$. 
	By adjointness and the definition of Koszul cohomology it provides 
	an isomorphism 
	\[
	K_{\bullet}(\xx-\UU;X[[\UU]]) \cong \Hom_R(K^{\bullet}(\xx-\UU;R[\UU^{-1}]),X),
	\]
	which completes the proof.
\end{proof}

A second duality result is the following. It is a consequence of the 
previous statement.

\begin{corollary} \label{dual-2}
	With the previous notation there is an isomorphism 
	\[
	\Hom_R(K^{\bullet}(\xx-\UU;X[\UU^{-1}]),Y) \cong K_{\bullet}(\xx-\UU;\Hom_R(X,Y)[[\UU]])
	\]
	for any two $R$-complexes $X$ and $Y$.
\end{corollary}

\begin{proof} 
	We apply the isomorphism of \ref{dual-1} to the $R$-complex $\Hom_R(X,Y)$. 
	It provides an isomorphism 
	\[
	\Hom_R(K^{\bullet}(\xx-\UU;R[\UU^{-1}]),\Hom_R(X,Y)) \cong
	K_{\bullet}(\xx-\UU;\Hom_R(X,Y)[[\UU]]).
	\]
	For the proof we have to investigate the first complex. By adjointness 
	there is an isomorphism 
	\[
	\Hom_R(K^{\bullet}(\xx-\UU;R[\UU^{-1}]),\Hom_R(X,Y)) \cong 
	\Hom_R(K^{\bullet}(\xx-\UU;R[\UU^{-1}]) \otimes_R X,Y).
	\]
	By view of \ref{coh-1} there is an isomorphism $R[\UU^{-1}] \otimes_R X \cong X[\UU^{-1}]$. Whence we get the following isomorphism 
	$K^{\bullet}(\xx-\UU;R[\UU^{-1}] \otimes_R X \cong 
	K^{\bullet}(\xx-\UU;X[\UU^{-1}])$. So it follows that 
	\[
	\Hom_R(K^{\bullet}(\xx-\UU;R[\UU^{-1}]) \otimes_R X,Y) \cong 
	\Hom_R(K^{\bullet}(\xx-\UU;X[\UU^{-1}]),Y),
	\]
	which completes the proof.
\end{proof}

Another corollary is the following statement.

\begin{corollary} \label{dual-3}
	With the previous notation  there is an isomorphism 
	\[
	K_{\bullet}(\xx-\UU;\Hom_R(X,Y)[[\UU]]) \cong \Hom_R(X,K_{\bullet}(\xx-\UU;Y[[\UU]]))
	\]
	for any two $R$-complexes $X$ and $Y$.
\end{corollary}

\begin{proof}
	As shown in the proof of \ref{dual-2} the complex on the left 
	is quasi-isomorphic to 
	\[
	\Hom_R(K^{\bullet}(\xx-\UU;R[\UU^{-1}]),\Hom_R(X,Y)) \cong 
	\Hom_R(X, \Hom_R(K^{\bullet}(\xx-\UU;R[\UU^{-1}]),Y))
	\]
	as follows by adjointness. It was shown in \ref{dual-1} that 
	there is an isomorphism 
	\[
	\Hom_R(K^{\bullet}(\xx-\UU;R[\UU^{-1}]),Y) \cong
	K_{\bullet}(\xx-\UU;Y[[\UU]]).
	\]
	This induces the isomorphism of the statement.
\end{proof}

%In the next we shall look at a counterpart of \ref{weak-2} for the local cohomology.
%
%\begin{theorem} \label{dual-4}
%	Let $\xx = x_1,\ldots,x_r$ denote a weakly pro-regular sequence 
%	and $\mathfrak{a} = \xx R$. Let $X$ denote an $R$-complex. Then there are isomorphisms 
%	\[
%	H^i_{\mathfrak{a}}(X) \cong H^i(\xx-\UU;X[\UU^{-1}]) 
%	\]
%	for all $i \in \mathbb{Z}$. More general, the Koszul complex $K^{\bullet}(\xx-\UU;X[\UU^{-1}])$ is a
%	representative of ${\rm R} \Gamma_{\mathfrak{a}}(X)$ in the derived 
%	category.
%\end{theorem}
%
%\begin{proof}
%	If $\xx$ is a weakly pro-regular sequence, then  
%	isomorphisms 
%	\[
%	H^i_{\mathfrak{a}}(X) \cong H^i(\xx-\UU;X[\UU^{-1}]) 
%	\]
%	for all $i \in \mathbb{Z}$ (see \ref{coh-6}). 
%	
%	Therefore, By \ref{coh-3}
%	there is a quasi-isomorphism $\check{C}_{\xx}(X) \simeq K^{\bullet}(\xx-\UU;X[\UU^{-1}])$. 
%	This proves all of the statements.
%\end{proof}

As a corollary we get the expression of local cohomology related to a certain 
dual to the statement shown in \ref{weak-9}.

\begin{corollary} \label{dual-5}
	Let $\xx = x_1,\ldots,x_r$ denote a weakly pro-regular sequence 
	and $\mathfrak{a} = \xx R$. Let $X$ denote an $R$-complex. Then there are isomorphisms 
	\[
	H^i_{\mathfrak{a}}(X) \cong 
	\Tor^{R[\UU]}_{r-i}(R[\UU]/(\xx-\UU)R[\UU], X[\UU^{-1}]) \cong \Ext_{R[\UU]}^i(R[\UU]/(\xx-\UU)R[\UU], X[\UU^{-1}])
	\]
	for all $i \in \mathbb{Z}$.
\end{corollary}

\begin{proof}
	If $\xx$ is a weakly pro-regular sequence, then  there are 
	isomorphisms 
	\[
	H^i_{\mathfrak{a}}(X) \cong H^i(\xx-\UU;X[\UU^{-1}]) 
	\]
	for all $i \in \mathbb{Z}$ (see \ref{coh-6}). 
	
	Then the statements are consequences of 
	\ref{coh-5}.
\end{proof}

Another duality statement is the following.

\begin{theorem} \label{dual-6}
	Let $\xx = x_1,\ldots,x_r$ denote a system of elements of $R$. Then there are 
	isomorphisms of complexes
	\[
	\Hom_R(K^{\bullet}(\xx-\UU;X[\UU^{-1}]),Y) \cong \Hom_{R[\UU]}(K^{\bullet}(\xx-\UU;X[\UU]),Y[[\UU]]) \cong 
	K_{\bullet}(\xx-\UU;\Hom_R(X,Y)[[\UU]])
	\]
	for all $R$-complexes $X,Y$.
\end{theorem}

\begin{proof}
	By \ref{coh-1} (C) we have that $X[\UU^{-1}] \cong R[\UU^{-1}] \otimes_{R[\UU]} 
	X[\UU]$. Whence the fist complex is isomorphic to 
	\[
	\Hom_R(K^{\bullet}(\xx-\UU;R[\UU^{-1}] \otimes_{R[\UU]} 
	X[\UU]),Y) \cong \Hom_{R[\UU]}(K^{\bullet}(\xx-\UU;X[\UU]), \Hom_R(R[\UU^{-1}],Y))
	\]
	as follows by adjointness. Because of $\Hom_R(R[\UU^{-1}],Y) \cong Y[[\UU]]$ 
	(see \ref{dual-0}) this proves the first isomorphism. 
	
	Because of adjointness the first complex of the statement is isomorphic to 
	\[
	\Hom_{R[\UU]}(K^{\bullet}(\xx-\UU;R[\UU]), \Hom_R(X[\UU^{-1}],Y)) 
	\cong K_{\bullet}(\xx-\UU; \Hom_R(X[\UU^{-1}],Y))
	\]
	(see e.g. \cite[5.2.3]{SS}). But now there are isomorphisms 
	\[
	\Hom_R(X[\UU^{-1}],Y) \cong \Hom_R(R[\UU^{-1}] \otimes_R X,Y) \cong 
	\Hom_R(R[\UU^{-1}],\Hom_R(X,Y)) \cong \Hom_R(X,Y)[[\UU]]
	\]
	as they follow by adjointness and by \ref{dual-0}. Putting together this 
	proves the second isomorphism of the statement.
\end{proof}

Next we shall provide a certain kind of duality for a certain dual 
of the \v{C}ech cohomology. 

\begin{theorem} \label{dual-7}
	Let $\xx = x_1,\ldots,x_r$ denote a system of elements of a commutative 
	ring $R$. Let $X$ be an $R$-complex and let $I$ be an injective $R$-module. 
	then there are isomorphisms 
	\[
	\Hom_R(H^i(\check{C}_{\xx}(X)), I) \cong \Ext_{R[\UU]}^{r-i}(X[\UU]/{(\xx-\UU)}X[\UU], I[[\UU]])
	\]
	for all $i \in \mathbb{Z}$. 
\end{theorem}

\begin{proof}
	By virtue of \ref{coh-5} there are isomorphisms 
	\[
	\Hom_R(H^i(\check{C}_{\xx}(X)),I) \cong \Hom_R(\Tor_{r-i}^{R[\UU]}(X[\UU]/(\xx-\UU)X[\UU],R[\UU^{-1}]),I)
	\]
	for all $i \in \mathbb{Z}$. The second modules are isomorphic to 
	\[
	\Ext_{R[\UU]}^{r-i}(X[\UU]/(\xx-\UU)X[\UU], \Hom_R(R[\UU^{-1}],I))
	\]
	as follows by the Ext-Tor duality (see e.g. \cite[1.4.1]{SS}). 
	Now (see \ref{dual-0}) there is an $R[\UU]$-isomorphism $\Hom_R(R[\UU^{-1}],I) \cong I[[\UU]]$ 
	which finishes the proof.
\end{proof}

In the following we shall give an application to weakly pro-regular sequences. 

\begin{corollary} \label{dual-8} 
	Let $\xx = x_1,\ldots,x_r$ denote a weakly pro-regular sequence of $R$ 
	and $\mathfrak{a} = \xx R$. Let $M$ an $R$-module and let $I$ an injective 
	$R$-module. Then there are natural isomorphisms 
	\[
	\Hom_R(H^i_{\mathfrak{a}}(M),I) \cong \Ext_{R[\UU]}^{r-i}(M[\UU]/(\xx-\UU)M[\UU], I[[\UU]])
	\]
	for all $i \in \mathbb{Z}$. 
\end{corollary}

\begin{proof}
	In case $\xx = x_1,\ldots,x_r$ is a weakly pro-regular sequence it 
	follows that $H^i_{\mathfrak{a}}(M) \cong H^i(\check{C}_{\xx}(M))$ 
	for all $i \in \mathbb{Z}$ (see e.g. \cite[7.4.4]{SS}). Then the statement 
	follows by virtue of \ref{dual-5}.
\end{proof}

\begin{remark} \label{dual-9}
	(A) Let $(R,\mathfrak{m})$ denote a local Gorenstein ring with $E = E_R(R/\mathfrak{m})$ the injective hull of the residue field (see \cite{Bh} for the definition). Let 
	$\xx = x_1,\ldots,x_r$ denote a system of parameters of $R$. Let 
	$M$ be an $R$-module. Then there are natural isomorphisms 
	\[
	H^i_{\mathfrak{m}}(M) \cong \Tor_{d-i}^R(M,E)
	\]
	for all $i \geq 0$. This follows since $\check{C}_{\xx}$ is a flat 
	resolution of $E$ (see e.g. \cite{SS}). By the Ext-Tor duality it 
	provides isomorphisms 
	\[
	\Hom_R(H^i_{\mathfrak{m}}(M),E) \cong 
	\Hom_R(\Tor_{d-i}^R(M,E),E) \cong \Ext_R^{d-i}(M,\hat{R})
	\]
	This is a special case of the local duality (see e.g. \cite[10.5.5]{SS})\\
	(B) Let $(R,\mathfrak{m})$ denote a local Noetherian ring and $d = \dim R$. 
	By view of the result in \ref{dual-8} there is an isomorphism 
	\[
	\Hom_R(H^i_{\mathfrak{m}}(M),E) \cong 
	\Ext_{R[\UU]}^{d-i}(M[\UU]/(\xx-\UU)M[\UU], E[[\UU]])
	\]
	for all $i \geq 0$, where $\xx = x_1,\ldots,x_d$ denotes a system of parameters and $\UU = U_1,\ldots,U_d$. It follows that \ref{dual-8} 
	is -- in a certain sense -- an extension of 
	the local duality theorem of a Gorenstein ring. A more detailed discussion 
	about these dualities are in preparation.
\end{remark}

\section{Enlargments of Sequences}

In this Section we shall use properties of Koszul complexes 
in order to show the behaviour of Koszul (co-)homology by enlarging the system of 
elements.

\begin{lemma} \label{enl-1}
	Let $\xx = x_1,\ldots,x_r$ denote a system of elements of $R$ and $y \in R$.
	Let $\UU = U_1,\ldots,U_r,V$ denote a system of variables over $R$. 
	Let $X$ denote an $R$-complex. Then there are short exact sequences 
	\begin{itemize}
		\item[(a)] 
		\begin{gather*}
		0 \to H_0(y-V;H_i(\xx-\UU;X[[\UU]])[[V]]) \to H_i(\xx-\UU,y-V;X[[\UU,V]]) 
		\\ \to H_1(y-V;H_{i-1}(\xx-\UU;X[[\UU]])[[V]]) \to 0, 
		\end{gather*}
		\item[(b)] 
		\begin{gather*}
		0 \to H^1(y-V;H^{i-1}(\xx-\UU;X[[\UU^{-1}]])[[V^{-1}]]) \to H^i(\xx-\UU,y-V;X[[(\UU,V)^{-1}]]) 
		\\ \to H^0(y-V;H^i(\xx-\UU;X[[\UU^{-1}]])[[V^{-1}]]) \to 0
		\end{gather*}
	\end{itemize}
	for all $i \in \mathbb{Z}$.
\end{lemma}

\begin{proof}
	We proof the statement in (a). The proof of (b) follows similar arguments. 
	For a system of elements $\ts = t_1,\ldots,t_r$ and $s \in R$ there is a 
	short exact sequence of Koszul homology
	\[
	0 \to H_0(s;H_i(\ts;Y)) \to H_i(\ts,s;Y) \to H_1(s;H_{i-1}(\ts;Y)) \to 0
	\] 
	for all $i \in \mathbb{Z}$ and an $R$-complex $Y$ (see e.g. \cite[5.2.4]{SS}).
	We apply this sequence to $\ts = \xx-\UU, s= y-V$ and $Y = X[[\UU,V]]$. In 
	order to conclude note that 
	$$
	H_i(\xx-\UU;X[[\UU,V]]) \cong H_i(\xx-\UU;X[[\UU]])[[V]] 
	$$ 
	for all $i \in \mathbb{Z}$. This follows by view of the corresponding Kozul 
	complex because homology commutes with direct products.  
\end{proof}

As an application for local cohomology and derived functors of the completion we get the following. 

\begin{proposition} \label{enl-2}
	Let $\xx = x_1,\ldots,x_k$ denote a weakly pro-regular sequence
	in a commutative ring $R$ and put
	$\mathfrak{a} = \xx R$. Suppose that $R$ is of bounded $yR$-torsion
	and such that $\xx,y$ is a weakly pro-regular sequence too. Let $X$
	denote a complex of $R$-modules. Then there are short exact sequences
	\begin{itemize}
		\item[(a)] $0 \to \Lambda_0^{yR}(\Lambda_i^{\mathfrak{a}}(X) )\to
		\Lambda_i^{(\mathfrak{a}+yR)}(X)
		\to \Lambda_1^{yR}(\Lambda_{i-1}^{\mathfrak{a}}(X) )\to 0$ and
		\item[(b)] $0 \to H^1_{yR}(H^{i-1}_{\mathfrak{a}}(X)) \to
		H^i_{(\mathfrak{a}+yR)}(X) \to H^0_{yR}(H^i_{\mathfrak{a}}(X)) \to 0$ 
	\end{itemize}
	for all $i \in \mathbb{Z}$.
\end{proposition}

\begin{proof}
	The statements are consequences of \ref{enl-1} by view of \ref{weak-9} and 
	\ref{coh-6}.
\end{proof}

A different proof of \ref{enl-2} is given in \cite[9.4.6]{SS}. There the 
proof for the statement in (a) is based on a certain spectral sequence argument. 

\begin{remark} \label{enl-3}
	(A) With the notation of \ref{enl-2} the exact sequences in (b) are consequences 
	of the mapping cone construction of the Koszul complex. In fact it yields 
	a long exact cohomology sequence
	\[
	\cdots \to H^{i-1}_{\mathfrak{a}}(X) \to H^{i-1}_{\mathfrak{a}}(X)_y 
	\to H^i_{(\mathfrak{a}+yR)}(X) \to H^i_{\mathfrak{a}}(X) \to 
	H^i_{\mathfrak{a}}(X)_y \to \cdots,
	\]
	where $\iota : H^i_{\mathfrak{a}}(X) \to H^i_{\mathfrak{a}}(X)_y$ is the natural map to the localization. This is well-known, see e.g. \cite[8.1.2]{BrS} for a proof in the case of $X$ an $R$-module.\\
	(B)
	It was a natural question to the author whether there is a corresponding 
	long exact sequence for the statement in \ref{enl-2} (a). Here is an 
	answer: there is a long exact sequence as follows by the mapping cone construction 
	\[
	\cdots \to \Lambda_i^{\mathfrak{a}}(X)[[V]] \stackrel{y-V}{\longrightarrow} 
	\Lambda_i^{\mathfrak{a}}(X)[[V]] \to \Lambda_i^{(\mathfrak{a}+yR)}(X) 
	\to \Lambda_{i-1}^{\mathfrak{a}}(X)[[V]] \stackrel{y-V}{\longrightarrow} 
	\Lambda_{i-1}^{\mathfrak{a}}(X)[[V]] \to \cdots
	\] 
	Note that $\Lambda_i^{\mathfrak{a}}(X)[[V]] \stackrel{y-V}{\longrightarrow} 
	\Lambda_i^{\mathfrak{a}}(X)[[V]]$ is the Koszul complex $K_{\bullet}(y-V;\Lambda_i^{\mathfrak{a}}(X)[[V]])$. Then, since $y$ is of bounded $R$-torsion we get 
	\[
	\Lambda_j^{yR}(\Lambda_i^{\mathfrak{a}}(X))
	\cong H_j(y-V;\Lambda_i^{\mathfrak{a}}(X)[[V]])) 
	\]
	as follows by virtue of \ref{weak-9}. \\
	(C) Therefore there is also another family of short exact sequences 
	\[
	0 \to \Lambda_i^{\mathfrak{a}}(X)[[V]]/(y-V)\Lambda_i^{\mathfrak{a}}(X)[[V]] \to 
	\Lambda_i^{(\mathfrak{a}+yR)}(X) \to 0:_{\Lambda_{i-1}^{\mathfrak{a}}(X)[[V]]} (y-V) 
	\to 0
	\]
	for all $i \in \mathbb{Z}$.
\end{remark}

\begin{remark} \label{enl-4}
	Let $y \in R$ be a weakly pro-regular element, i.e. $R$ is of bounded $yR$-torsion. 
	Let $M$ be an $R$-module. Then there is a five term short exact sequence 
	\[
	0\to \Lambda_1^{yR}(M) \to \Hom_R(R_y, M) \to M \to \Lambda_0^{yR}(M) \to \Ext^1_R(R_y, M) \to 0
	\]
	(see \cite[8.4.6]{SS}). As an application we get an approximation of $\Lambda_j^{yR}(\Lambda_i^{\mathfrak{a}}(X)), j = 0,1,$
	namely there is an exact sequence
	\[
	0\to \Lambda_1^{yR}(\Lambda_i^{\mathfrak{a}}(X)) \to \Hom_R(R_y, \Lambda_i^{\mathfrak{a}}(X)) \to \Lambda_i^{\mathfrak{a}}(X) \to \Lambda_0^{yR}(\Lambda_i^{\mathfrak{a}}(X)) \to \Ext^1_R(R_y, \Lambda_i^{\mathfrak{a}}(X)) \to 0
	\]
	for all $i \in \mathbb{Z}$.
\end{remark}

This provides us to the following.

\begin{corollary} \label{enl-5}
	Fix the assumptions of \ref{enl-2} and let $i \in \mathbb{Z}$. Suppose that $\Lambda_i^{\mathfrak{a}}(X)$ is $yR$-complete. 
	\begin{itemize}
		\item[(a)] $\Lambda_0^{yR}(\Lambda_i^{\mathfrak{a}}(X)) \cong \Lambda_i^{\mathfrak{a}}(X)$ and $\Lambda_1^{yR}(\Lambda_i^{\mathfrak{a}}(X)) = 0$.
		\item[(b)] There is a short exact sequence
		$0 \to \Lambda_i^{\mathfrak{a}}(X) \to \Lambda_i^{(\mathfrak{a}+yR)}(X)
		\to \Lambda_1^{yR}(\Lambda_{i-1}^{\mathfrak{a}}(X) )\to 0$.
		\item[(c)] If in addition $\Lambda_{i-1}^{\mathfrak{a}}(X) $ is $yR$-complete, 
		then $\Lambda_i^{\mathfrak{a}}(X) \cong \Lambda_i^{(\mathfrak{a}+yR)}(X)$.
	\end{itemize}
\end{corollary}

\begin{proof}
	If it is $yR$-complete, then 
	\[
	\Hom_R(R_y,\Lambda_j^{\mathfrak{a}}(X)) = \Ext_R^1(R_y,\Lambda_j^{\mathfrak{a}}(X)) = 0
	\]
	(see \cite{Sp3} or \cite[3.1.9]{SS} for the details). This proves (a) 
	by view of \ref{enl-4}.  
	
	If $\Lambda_j^{\mathfrak{a}}(X)$ is $yR$-separated, then $\Hom_R(R_y,\Lambda_j^{\mathfrak{a}}(X)) = 0$.  Therefore the result in (b) is 
	a consequences of \ref{enl-4} and \ref{enl-2}.
	
	Finally (c) is a consequence of (b).
\end{proof}
%{\Huge $\operatorname{Bl}_U (x,y)$ \hspace*{2cm} $\tilde{X}$ \hspace*{2cm} $\to$
%\hspace*{2cm} $\longrightarrow$}

%\include{ch07b-180104}

\bibliographystyle{siam}

\bibliography{cech}

\end{document}